\documentclass[reqno, eucal]{amsart}

\usepackage{epic,eepic}
\usepackage[mathscr]{eucal} 
\usepackage{psfrag}     

\usepackage{amsmath}
\usepackage{amsthm}
\usepackage[dvips]{color}

\usepackage[dvips]{graphicx,color}  

\usepackage{amssymb}
\usepackage{epic,eepic}
\usepackage{amscd}
\usepackage{longtable}
\usepackage{array}

\newtheorem{theorem}{Theorem}
\newtheorem{lemma}[theorem]{Lemma}

\newtheorem{proposition}[theorem]{Proposition}

\theoremstyle{definition}

\newtheorem{example}[theorem]{Example}
\newtheorem{remark}[theorem]{Remark}

\renewcommand{\labelenumi}{(\roman{enumi})}

\newcommand{\bk}{{\bf k}}
\newcommand{\bp}{{\bf b}_+}
\newcommand{\bm}{{\bf b}_-}
\newcommand{\bpm}{{\bf b}_\pm}

\newcommand{\g}{g}

\newcommand{\Z}{{\mathbb Z}}
\newcommand{\R}{{\mathbb R}}
\newcommand{\C}{{\mathbb C}}

\begin{document}

\title[$U_q(A^{(1)}_2)$-zero range process]
{Matrix product formula for $\boldsymbol{U_q(A^{(1)}_2)}$-zero range process}
\author{Atsuo Kuniba}
\email{atsuo@gokutan.c.u-tokyo.ac.jp}
\address{Institute of Physics, University of Tokyo, Komaba, Tokyo 153-8902, Japan}

\author{Masato Okado}
\email{okado@sci.osaka-cu.ac.jp}
\address{Department of Mathematics, Osaka City University, 
3-3-138, Sugimoto, Sumiyoshi-ku, Osaka, 558-8585, Japan}


\maketitle

\begin{center}{\bf Abstract}\end{center}

The $U_q(A^{(1)}_n)$-zero range processes 
introduced recently 
by Mangazeev, Maruyama and the authors are integrable 
discrete and continuous time Markov processes 
associated with the stochastic $R$ matrix
derived from the well-known $U_q(A_n^{(1)})$ quantum $R$ matrix.
By constructing a representation of the relevant 
Zamolodchikov-Faddeev algebra,
we present, for  $n=2$, a matrix product formula 
for the steady state probabilities
in terms of $q$-boson operators.
\vspace{0.3cm}

\section{Introduction and main result}\label{sec:int}

Zero range processes \cite{S} are stochastic dynamical models 
for a variety of systems 
in biology, chemistry, networks, physics, sociology, traffic flows and so forth.
Investigating their rich behaviors like 
condensation, current fluctuations and hydrodynamic limit, etc 
has been a prominent theme in mathematical physics of 
non-equilibrium phenomena. 
See for example \cite{EH, GSS, KL} and the references therein.

In the recent work \cite{KMMO},
new integrable Markov processes 
associated with the quantum affine algebra $U_q(A^{(1)}_n)$ \cite{D,J}
have been constructed.
They are described naturally in terms of discrete and continuous time 
stochastic dynamics of $n$-species of particles on 
a ring obeying a zero range type interaction.
We call them $U_q(A^{(1)}_n)$-zero range processes (ZRPs) in this paper.
Here is a snapshot of the system for the $n=2$ case on the $L$ site periodic chain:

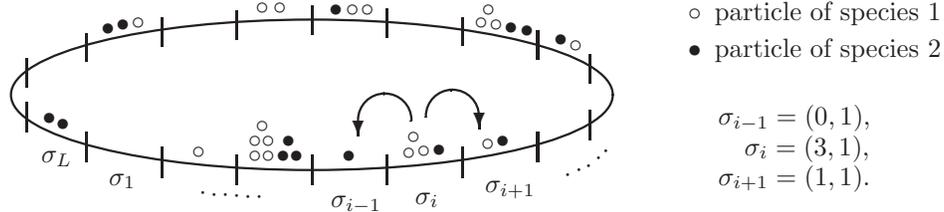
\setlength{\unitlength}{1mm}
\begin{picture}(100,35)(-5,33)
\thicklines
\put(50,50){\ellipse{80}{20}}

\put(12,51){\line(0,1){3.8}} 
\put(20,54.4){\line(0,1){3.8}}
  \put(22,58){$\bullet$}\put(24,58.5){$\bullet$}\put(26,58.8){$\circ$}

\put(30,56.4){\line(0,1){3.8}}
\put(40,57.4){\line(0,1){3.8}}
   \put(42.5,60.7){$\circ\,\circ$}

\put(50,58){\line(0,1){3.8}}
  \put(52.3,60.5){$\bullet$} \put(54.5,60.5){$\circ$}\put(56.5,60.5){$\circ$}
  
\put(60,57.6){\line(0,1){3.8}}

\put(70,56.5){\line(0,1){3.8}}
  \put(71.6,58.7){$\circ$}\put(73.6,58.5){$\circ$}\put(75.6,58.1){$\bullet$}
  \put(72.6,60.5){$\circ$}\put(77.7,57.6){$\bullet$}

\put(80,54.7){\line(0,1){3.8}}
  \put(82.2,56.5){$\bullet$}\put(84.2,55.7){$\circ$}

\put(87,51.5){\line(0,1){3.8}}

\put(12,45){\line(0,1){4}}
  \put(14.2,46){$\bullet$}\put(16,45.3){$\bullet$}

\put(20,41){\line(0,1){4}}
\put(30,39){\line(0,1){4}}
  \put(34,41.5){$\circ$}

\put(35,36.1){$\cdot$}
\put(36.5,35.9){$\cdot$}
\put(38,35.8){$\cdot$}
\put(39.5,35.7){$\cdot$}
\put(41,35.6){$\cdot$}
\put(42.5,35.6){$\cdot$}

\put(40,38){\line(0,1){4}}
\put(-0.5,0){
 \put(42,41){$\circ \circ$}\put(42,43){$\circ \circ$}\put(43,45){$\circ$}
 \put(45.8,41){$\bullet \bullet$}\put(46.5,43){$\bullet$}}

\put(50,38){\line(0,1){4}}
  \put(54,41){$\bullet$}

\put(60,38){\line(0,1){4}}
  \put(62,41.4){$\circ$}\put(64,41.6){$\circ$}
  \put(62.6,43.5){$\circ$}\put(66,41.8){$\bullet$}

\put(70,39){\line(0,1){4}}
  \put(72.5,42.6){$\circ$}\put(74.5,43){$\bullet$}

\put(80,41){\line(0,1){4}}
\put(87,44){\line(0,1){4}}

\put(14,41){$\sigma_L$}
\put(23,38){$\sigma_1$}
\put(52.5,35){$\sigma_{i-1}$}
\put(63.7,35.5){$\sigma_i$}
\put(73,37){$\sigma_{i+1}$}

\put(83.5,38.5){$\cdot$}
\put(85,39.1){$\cdot$}
\put(86.3,39.7){$\cdot$}
\put(87.6,40.7){$\cdot$}
\put(88.6,41.7){$\cdot$}

\put(59.5,46.5){\oval(7,7)[t]}\put(56.1,45.5){\vector(0,-1){1}}
\put(68.5,47){\oval(7,7)[t]}\put(72.2 ,46.5){\vector(0,-1){1}}

\put(0,-6){
\put(100,66){$\circ\,$ particle of species 1}
\put(100,61){$\bullet\,$ particle of species 2}

\put(104,52){$\sigma_{i-1}=(0,1)$,}
\put(107.5,48){$\sigma_i=(3,1)$,}
\put(104,44){$\sigma_{i+1}=(1,1)$.}
}

\end{picture}
\setlength{\unitlength}{1pt}

\noindent
For $n$ general, a local state at site $i \in \Z_L$ is an array
$\sigma_i=(\sigma_{i,1},\ldots, \sigma_{i,n}) 
\in \Z^n_{\ge 0}$ signifying that 
there are $\sigma_{i,s}$ particles of species $s\,(1 \le s \le n)$.
There is no constraint on the particles occupying a site.
In the continuous time version of the model,
they can hop either to the right or to the left adjacent sites with 
a zero range interaction, which means that
the local transition rate depends on the occupancy of the 
departure site only and not on the destination site.
For the $U_q(A^{(1)}_2)$-ZRP, 
the rate for the hopping of 
$\gamma_1+\gamma_2$ particles is given 
(when $\gamma_1+\gamma_2\ge 1$) as

\setlength{\unitlength}{0.7mm}
\begin{picture}(200,72)(-5,-1)
\put(0,48){$\displaystyle{
a \frac{q^{(\alpha_1-\gamma_1)\gamma_2}\mu^{\gamma_1+\gamma_2-1}
(q)_{\gamma_1+\gamma_2-1}}
{(\mu q^{\alpha_1+\alpha_2-\gamma_1-\gamma_2})_{\gamma_1+\gamma_2}}
\frac{(q)_{\alpha_1}}{(q)_{\gamma_1}(q)_{\alpha_1-\gamma_1}}
\frac{(q)_{\alpha_2}}{(q)_{\gamma_2}(q)_{\alpha_2-\gamma_2}}}$ \;\;\;\;\;\;for
\put(10,-10){
\drawline(0,5)(0,0)
\drawline(0,0)(50,0)\drawline(50,0)(50,5)
\drawline(25,0)(25,5)
\put(-1,0){
\drawline(12,12)(12,17)\drawline(12,17)(37,17)\put(37,17){\vector(0,-1){5}}
}
\put(3,2){$\overbrace{\circ ... \circ}^{\alpha_1}
\overbrace{\bullet...\bullet}^{\alpha_2}$}
\put(28,2){$\circ ... \circ \bullet...\bullet$}
\put(14, 19){$\overbrace{\circ ... \circ}^{\gamma_1}
\overbrace{\bullet...\bullet}^{\gamma_2}$}}

}
\put(0,10){$\displaystyle{
b \frac{q^{\gamma_1(\beta_2-\gamma_2)}
(q)_{\gamma_1+\gamma_2-1}}
{(\mu q^{\beta_1+\beta_2-\gamma_1-\gamma_2})_{\gamma_1+\gamma_2}}
\frac{(q)_{\beta_1}}{(q)_{\gamma_1}(q)_{\beta_1-\gamma_1}}
\frac{(q)_{\beta_2}}{(q)_{\gamma_2}(q)_{\beta_2-\gamma_2}}}$
 \;\;\;\;\;\; \;\;\;\;\;\;\;\;\;\;\;\;for

\put(10,-8){
\drawline(0,5)(0,0)
\drawline(0,0)(50,0)\drawline(50,0)(50,5)
\drawline(25,0)(25,5)
%
\put(-1,0){
\drawline(37,12)(37,17)\drawline(37,17)(12,17)\put(12,17){\vector(0,-1){5}}
}
\put(28,2){$\overbrace{\circ ... \circ}^{\beta_1}
\overbrace{\bullet...\bullet}^{\beta_2}$}
\put(3,2){$\circ ... \circ \bullet...\bullet$}
\put(14, 19){$\overbrace{\circ ... \circ}^{\gamma_1}
\overbrace{\bullet...\bullet}^{\gamma_2}$}}

}
\end{picture}
\setlength{\unitlength}{1pt}

\noindent
Here $a,b,\mu, q$ are the parameters of the model and 
the symbol $(z)_m$ is defined in the end of this section.
These are the $n=2$ cases of (\ref{ctmp2})--(\ref{rate2}) with $\epsilon=1$.

The $U_q(A^{(1)}_n)$-ZRPs \cite{KMMO} contain the 
earlier proposed $n$ species models 
\cite{KMO1,KMO2,T} via various specialization of the parameters. 
In particular for $n=1$, they reproduce the single species models
studied in \cite{SW,P,T0,CP,BP} up to boundary conditions. 
A more detailed explanation is available in Section \ref{ss:sra}.

The above transition rate has been chosen so as to guarantee the 
integrability, or put more practically the Bethe ansatz solvability of the model
via the {\em stochastic $R$ matrix} $\mathscr{S}(\lambda,\mu)$ \cite{KMMO}.
It originates in the quantum $R$ matrix for 
the symmetric tensor representations of $U_q(A^{(1)}_n)$, a basic example of  
higher-spin representations of higher-rank quantum groups, 
in the special gauge that allows a probabilistic interpretation.
Its nonzero matrix elements 
$\mathscr{S}(\lambda,\mu)^{\gamma, \delta}_{\alpha, \beta}$ are 
described by the function 
$\Phi_q(\gamma | \beta; \lambda,\mu)$ 
$(\beta, \gamma \in \Z^n_{\ge 0} )$
defined in (\ref{mho}) as 
\begin{align*}
q^{\sum_{1 \le i < j \le n}(\beta_i-\gamma_i)\gamma_j}
\left(\frac{\mu}{\lambda}\right)^{\gamma_1+\cdots + \gamma_n}
\frac{(\lambda)_{\gamma_1+\cdots + \gamma_n}
(\frac{\mu}{\lambda})_{\beta_1+\cdots + \beta_n
-\gamma_1-\cdots - \gamma_n}}
{(\mu)_{\beta_1+\cdots + \beta_n}}
\prod_{i=1}^{n}\frac{(q)_{\beta_i}}
{(q)_{\gamma_i}(q)_{\beta_i-\gamma_i}}.
\end{align*}
For $n=1$ it reduces to the transition rate in the chipping model \cite{P}, which
was also built in the explicit formulas 
of the $R$ matrix and $Q$ operators for $U_q(A^{(1)}_1)$ \cite{M1}.
The parameters $\lambda$ and $\mu$ 
are reminiscents of the degrees of the relevant 
symmetric tensor representations of $U_q(A^{(1)}_n)$.
They play a role analogous to the 
spectral parameter although the difference property
$\mathscr{S}(c\lambda,c\mu) = \mathscr{S}(\lambda, \mu)$ is {\em absent}.

As the usual vertex models in equilibrium statistical mechanics \cite{Bax}, 
the stochastic $R$ matrix $\mathscr{S}(\lambda, \mu_i)$
serves as a building block of the commuting {\em Markov transfer matrix}
$T(\lambda|\mu_1,\ldots, \mu_L)$ (\ref{ngm})
for the discrete time process. 
It is governed by the master equation
\begin{align*}
|P(t+1)\rangle = T(\lambda|\mu_1,\ldots, \mu_L)
|P(t)\rangle
\end{align*}
with the time variable $t$,
where $\mu_i$ is the inhomogeneity
assigned to each lattice site $i \in \Z_L$.
The discrete time ZRP covers the 
continuous time one in the sense that 
the Markov matrix of the latter is derived from the 
homogeneous case $T(\lambda|\mu,\ldots, \mu)$ of the former 
by the logarithmic derivative (\ref{bax})
as in the well-known Baxter's formula 
for spin chain Hamiltonians \cite[Sec.10.14]{Bax}.
A novel feature of the $U_q(A^{(1)}_n)$ 
Markov transfer matrix $T(\lambda|\mu,\ldots, \mu)$ is the presence of 
{\em two} natural ``Hamiltonian points" $\lambda=1$ and $\lambda=\mu$
yielding the two continuous time Markov matrices 
$H^{(1)}$ and $H^{(2)}$ in (\ref{ctmp2}).
They give rise to the right and the left moving 
particles whose {\em mixture} is still integrable thanks to $[H^{(1)}, H^{(2)}]=0$. 
These aspects have been demonstrated in detail in \cite[Sec.3.4]{KMMO}. 

In this paper we study the steady states of the 
$U_q(A^{(1)}_n)$-ZRPs.
By definition, steady states are those $|\overline{P}\rangle$ satisfying
$|\overline{P}\rangle= T(\lambda|\mu_1,\ldots, \mu_L)|\overline{P}\rangle$.
It exists uniquely in each sector specified 
by the total number of particles of each species, and serves as a basic characteristics
of the system analogous to the ground states in the equilibrium spin chain models.
Let ${\mathbb P}(\sigma_1,\ldots, \sigma_L)$ be the probability of finding the 
system in the configuration $(\sigma_1,\ldots, \sigma_L) \in (\Z^n_{\ge 0})^L$ 
in a steady state up to an overall normalization. 
For $n=1$, it is known to become a product of 
on-site (albeit inhomogeneous) factors \cite{P,EMZ} as 
\begin{align*}
{\mathbb P}(\sigma_1,\ldots, \sigma_L)
= \prod_{i=1}^L \g_{\sigma_i}(\mu_i),
\end{align*}
where $g_{\sigma_i}(\mu)$ is defined by (\ref{mgd}) for 
$\sigma_i \in \Z_{\ge 0}^n$ with general $n$.
Such a factorization, however, is
no longer valid in the multispecies case 
$n\ge 2$ as observed in \cite[Example 13]{KMMO},
and this becomes a source of interest in the present model even 
without an introduction of a reservoir; 
particles of a given species must behave under the influence of 
the other species ones acting as a nontrivial dynamical background.  
Our main result in this paper is the following matrix product formula for $n=2$:
\begin{equation}\label{okne}
\begin{split}
&{\mathbb P}(\sigma_1,\ldots, \sigma_L)
= \mathrm{Tr}(X_{\sigma_1}(\mu_1)\cdots X_{\sigma_L}(\mu_L)),
\\
&X_\alpha(\mu) = \g_\alpha(\mu) Z_\alpha(\mu),
\quad 
Z_\alpha(\mu) = \left(\, \prod_{m=0}^\infty
\frac{1-q^m\bp}{1-q^m\mu^{-1}\bp}\right)
\bk^{\alpha_2}\bm^{\alpha_1}. 
\end{split}
\end{equation}
Here $\alpha=(\alpha_1, \alpha_2)\in \Z^2_{\ge 0}$ and 
$\bp, \bm, \bk$ are $q$-boson operators (\ref{akn}) acting
on the Fock space over which the trace is to be taken.
One sees that the $n=1$ case formally corresponds to setting $Z_\alpha(\mu)=1$.
In this sense the above $Z_\alpha(\mu)$ is capturing    
the first multispecies effect beyond $n=1$, and it has been identified with  
a quantum dilogarithm or 
a single mode $q$-boson vertex operator
$Z_\alpha(\mu) 
= \exp\left(\sum_{l\ge 1}\frac{\mu^{-l}-1}{l(1-q^l)}\bp^l \right)
\bk^{\alpha_2}\bm^{\alpha_1}$.
It is curiously asymmetric with respect to $\alpha_1$ and $\alpha_2$
despite that the original transition rate looks fairly symmetric 
between the two species.  

There are many matrix product formulas in terms of bosons
known in the literature for similar models typically like exclusion processes.
See \cite{AL, BE, CRV, KMO0, PEM} for example and the references therein.
Our formula (\ref{okne}) is the first example distinct from them involving  
an {\em infinite product} of $q$-bosons.

In order to establish the above result,  we
invoke the so called {\em Zamolodchikov-Faddeev (ZF) algebra} \cite{Z,F}
having the stochastic $R$ matrix as the structure function. 
It reads symbolically as
\begin{align*}
X(\mu) \otimes X(\lambda) = 
\check{\mathscr{S}}(\lambda, \mu)\bigl[
X(\lambda) \otimes X(\mu)\bigr],
\end{align*}
where $\check{\mathscr{S}}(\lambda, \mu)$ is defined after (\ref{psum}).
See (\ref{msk}) for the concrete description.
It is a local version of the stationary condition, and 
plays a central role in deriving the matrix product formula.
Besides the ZF algebra however, 
we need one further essential ingredient 
which we call the {\em auxiliary condition} (\ref{koi})
on the operator $X_\alpha(\mu)$.
Its formulation is another important result in this paper which 
contrasts with the simpler situation of continuous time models on a ring 
(cf. \cite{SW0, AL, CRV, KMO0})
where the ZF algebra alone sufficed together with its derivative.
We find in the proof of Proposition \ref{pr:srn} that 
the $U_q(A^{(1)}_n)$ stochastic $R$ matrix fits the auxiliary condition perfectly
by an intriguing mechanism.
Consequently the matrix product construction for the arbitrary $n$
and the inhomogeneity 
$\mu_1,\ldots, \mu_L$ is attributed 
to the task of realizing such an operator $X_\alpha(\mu)$ concretely.
The result (\ref{okne}) is an outcome of this exercise for $n=2$.
The general $n$ case is also feasible and will be presented elsewhere. 

The outline of the paper is as follows.
In Section \ref{sec:ykn}
we recall the necessary facts on the $U_q(A^{(1)}_n)$-ZRPs in this paper.
The discrete and continuous time versions  
are those defined in section 3.3 and section 3.4 in \cite{KMMO}, respectively.
In Section \ref{sec:hnka} the matrix product formula for the steady state probabilities 
are linked with the ZF algebra (Proposition \ref{pr:srn}).
The auxiliary condition (\ref{koi}) or equivalently (\ref{aaa}) 
plays a key role.
Until this point all the arguments are valid for general $n$.
In Section \ref{sec:rna} we focus on the $n=2$ case
and present a concrete realization of the ZF algebra (Theorem \ref{th:nzm})
satisfying all the criteria in Proposition \ref{pr:srn}.
It leads to the matrix product formulae (\ref{mzs}) and (\ref{msm}),
which are the main results of the paper.
Section \ref{sec:tmn} contains a summary and discussion.
Systematic applications to the study of physical behaviors 
is a subject of a future research. 

Throughout the paper we use 
the notation
$\theta(\mathrm{true})=1, 
\theta(\mathrm{false}) =0$,
the $q$-Pochhammer symbol
$(z)_m = (z; q)_m = \prod_{j=1}^m(1-zq^{j-1})$
and the $q$-binomial 
$\binom{m}{k}_{\!q} = \theta(k \in [0,m])
\frac{(q)_m}{(q)_k(q)_{m-k}}$.
The symbols $(z)_m$ appearing in this paper always mean $(z; q)_m$.
For integer arrays 
$\alpha=(\alpha_1,\ldots, \alpha_m), \beta=(\beta_1,\ldots, \beta_m)$ 
of {\em any} length $m$, we write 
$|\alpha | = \alpha_1+\cdots  + \alpha_m$.  
The relation $\alpha \le \beta$ or equivalently  
$\beta \ge \alpha$ is defined by $\beta-\alpha \in \Z^m_{\ge 0}$. 
We often denote by $0$ to mean 
$(0, \ldots, 0) \in \Z^m_{\ge 0}$ for some $m$ when it is clear from the 
context.

\section{$U_q(A^{(1)}_n)$-zero range processes}\label{sec:ykn}
Let us briefly recall the stochastic $R$ matrix for $U_q(A^{(1)}_n)$ and the associated 
discrete and continuous time ZRPs constructed in 
\cite[Sec. 3.3, 3.4]{KMMO}.

\subsection{Stochastic $R$ matrix}
Set 
$W = \bigoplus_{\alpha=(\alpha_1,\ldots, \alpha_n) \in \Z_{\ge 0}^n}
\C |\alpha\rangle$.
Define the operator $\mathscr{S}(\lambda,\mu) \in \mathrm{End}(W \otimes W)$ 
depending on the parameters $\lambda$ and $\mu$ by
\begin{align}
&\mathscr{S}(\lambda,\mu)(|\alpha\rangle \otimes |\beta\rangle ) = 
\sum_{\gamma,\delta \in 
\Z_{\ge 0}^n}\mathscr{S}(\lambda,\mu)_{\alpha,\beta}^{\gamma,\delta}
\,|\gamma\rangle \otimes |\delta\rangle,
\label{ask1}\\
&\mathscr{S}(\lambda,\mu)^{\gamma,\delta}_{\alpha, \beta} 
= \theta(\gamma+\delta=\alpha+\beta)
\Phi_q(\gamma | \beta; \lambda,\mu), \label{ask2}
\end{align}
where $\Phi_q(\gamma | \beta; \lambda,\mu)$ is given by

\begin{align}
\Phi_q(\gamma|\beta; \lambda,\mu)  = 
q^{\varphi(\beta-\gamma, \gamma)} \left(\frac{\mu}{\lambda}\right)^{|\gamma|}
\frac{(\lambda)_{|\gamma|}(\frac{\mu}{\lambda})_{|\beta|-|\gamma|}}
{(\mu)_{|\beta|}}
\prod_{i=1}^{n}\binom{\beta_i}{\gamma_i}_{\!q},
\qquad
\varphi(\alpha, \beta) &= \sum_{1 \le i < j \le n}\alpha_i\beta_j. \label{mho}
\end{align}
The sum (\ref{ask1}) is finite due to the $\theta$ factor in (\ref{ask2}).
In fact the direct sum decomposition
$W \otimes W = \bigoplus_{\gamma \in \Z_{\ge 0}^n}
\left(\bigoplus_{\alpha+\beta=\gamma}\C
|\alpha\rangle \otimes |\beta\rangle\right)$ holds
and $\mathscr{S}(\lambda,\mu)$ splits into the corresponding submatrices. 
Note also that 
$\mathscr{S}(\lambda,\mu)^{\gamma,\delta}_{\alpha, \beta}=0$ 
unless $\gamma \le \beta$ and therefore $\alpha \le \delta$ as well.
The difference property 
$\mathscr{S}(\lambda,\mu)=\mathscr{S}(c\lambda,c\mu)$ is absent.
We call $\mathscr{S}(\lambda, \mu)$ 
the {\em stochastic $R$ matrix}.
Its elements are depicted as
\begin{equation}\label{vertex}
\begin{picture}(200,45)(-90,-21)
\put(-70,-2){$\mathscr{S}(\lambda,\mu)^{\gamma,\delta}_{\alpha, \beta}\,=$}
\put(0,0){\vector(1,0){24}}
\put(12,-12){\vector(0,1){24}}
\put(-10,-2){$\alpha$}\put(27,-2){$\gamma$}
\put(9,-22){$\beta$}\put(9,16){$\delta$}
\end{picture}
\end{equation}
The stochastic $R$ matrix originates in the 
quantum $R$ matrix of the symmetric tensor representation of 
the quantum affine algebra $U_q(A^{(1)}_n)$ \cite{D, J}.
It satisfies the Yang-Baxter equation, 
the inversion relation and the sum-to-unity condition \cite{KMMO}:
\begin{align}
&\mathscr{S}_{1,2}(\nu_1,\nu_2)
\mathscr{S}_{1,3}(\nu_1, \nu_3)
\mathscr{S}_{2,3}(\nu_2, \nu_3)
=
\mathscr{S}_{2,3}(\nu_2, \nu_3)
\mathscr{S}_{1,3}(\nu_1, \nu_3)
\mathscr{S}_{1,2}(\nu_1,\nu_2),
\label{sybe}\\
&\check{\mathscr{S}}(\lambda, \mu)
\check{\mathscr{S}}(\mu,\lambda)
= \mathrm{id}_{W^{\otimes 2}},
\label{sinv}\\
&\sum_{\gamma, \delta \in \Z_{\ge 0}^n}
\mathscr{S}(\lambda,\mu)^{\gamma,\delta}_{\alpha, \beta} =1\qquad
(\forall \alpha, \beta \in \Z_{\ge 0}^n),
\label{psum}
\end{align}
where the checked stochastic $R$ matrix is defined by 
$\check{\mathscr{S}}(\lambda,\mu)(|\alpha\rangle \otimes |\beta\rangle ) = 
\sum_{\gamma,\delta}
\mathscr{S}(\lambda,\mu)_{\alpha,\beta}^{\gamma,\delta}
\,|\delta\rangle \otimes |\gamma\rangle$.

Let $\mathscr{S}^T(\lambda,\mu)$ be the transpose of 
$\mathscr{S}(\lambda,\mu)$, i.e.,
\begin{align*}
\mathscr{S}^T(\lambda,\mu)(|\alpha\rangle \otimes |\beta\rangle ) = 
\sum_{\gamma,\delta \in 
\Z_{\ge 0}^n}\mathscr{S}(\lambda,\mu)^{\alpha,\beta}_{\gamma,\delta}
\,|\gamma\rangle \otimes |\delta\rangle.
\end{align*}
It satisfies the same Yang-Baxter equation as (\ref{sybe}):
\begin{align}
\mathscr{S}^T_{1,2}(\nu_1,\nu_2)
\mathscr{S}^T_{1,3}(\nu_1, \nu_3)
\mathscr{S}^T_{2,3}(\nu_2, \nu_3)
=
\mathscr{S}^T_{2,3}(\nu_2, \nu_3)
\mathscr{S}^T_{1,3}(\nu_1, \nu_3)
\mathscr{S}^T_{1,2}(\nu_1,\nu_2).
\label{tybe}
\end{align}
To see this, note the identity
\begin{align}
&\mathscr{S}(\lambda,\mu)_{\alpha,\beta}^{\gamma,\delta}=
\mathscr{S}(\lambda,\mu)^{\alpha',\beta'}_{\gamma',\delta'}
\frac{\tilde{\g}_\gamma(\lambda)\tilde{\g}_\delta(\mu)}
{\tilde{\g}_\alpha(\lambda)\tilde{\g}_\beta(\mu)}
q^{\varphi(\beta,\alpha)-\varphi(\gamma,\delta)},
\label{osr}\\
&
\tilde{\g}_\alpha(\mu) = \g_\alpha(\mu)q^{-\varphi(\alpha,\alpha)},
\qquad
\g_\alpha(\mu)=\frac{\mu^{-|\alpha|}(\mu)_{|\alpha|}}
{\prod_{i=1}^n(q)_{\alpha_i}},
\label{mgd}
\end{align}
where $\alpha'=(\alpha_n,\ldots, \alpha_1)$ is the reverse ordered array of 
$\alpha=(\alpha_1,\ldots, \alpha_n)$.
One can easily check that the extra factors in the RHS of (\ref{osr})
are gauge freedom not spoiling the Yang-Baxter equation.
The factor 
$\g_\alpha(\mu)$, which will appear frequently in the sequel, is a piece of 
the function $\Phi_q(\gamma | \beta; \lambda, \mu)$
in the following sense:
\begin{align}\label{gkn}
q^{\varphi(\beta,\gamma)}
\frac{\g_\beta(\mu) \g_\gamma(\lambda)}
{\g_{\beta+\gamma}(\mu)}
\mathscr{S}^{0,\alpha}_{\delta, \beta}(\lambda, \mu)
=\mathscr{S}^{\gamma,\alpha}_{\delta, \beta+\gamma}(\lambda, \mu).
\end{align}

The function 
$\Phi_q(\gamma | \beta; \lambda, \mu)|_{n=1}$ 
appeared earlier in \cite{M1, P}. 
For $n$ general it is zero unless $\gamma \le \beta$, and 
satisfies the sum rule \cite{KMMO}:
\begin{align}\label{syk}
\sum_{\gamma \in \Z_{\ge 0}^n}
\Phi_q(\gamma | \beta; \lambda,\mu)  = 1 \quad(\forall \beta \in \Z_{\ge 0}^n).
\end{align}
The relation (\ref{osr}) is a consequence of the property
\begin{align}\label{oys}
\frac{\g_\gamma(\lambda)\g_{\alpha+\beta-\gamma}(\mu)}
{\g_\alpha(\mu)\g_{\beta}(\lambda)}
\Phi_q(\beta|\alpha+\beta-\gamma;\lambda, \mu)
=q^{\varphi(\alpha-\gamma, \beta-\gamma)}
\Phi_q(\gamma|\alpha;\lambda, \mu).
\end{align}

\subsection{\mathversion{bold}Markov transfer matrix and discrete time 
$U_q(A^{(1)}_n)$-ZRP}\label{ss:sae}
Let $L$ be a positive integer.
Introduce the operator
\begin{align}\label{ngm}
T(\lambda|\mu_1,\ldots, \mu_L) = 
\mathrm{Tr}_{W}\left(
\mathscr{S}_{0,L}(\lambda,\mu_L)\cdots \mathscr{S}_{0,1}(\lambda,\mu_1)
\right)
\in \mathrm{End}(W^{\otimes L}).
\end{align}
In the terminology of the quantum inverse scattering method,
it is the row transfer matrix of the $U_q(A^{(1)}_n)$ vertex model 
of length $L$ with periodic boundary condition 
whose quantum space is  
$W^{\otimes L}$ with inhomogeneity parameters 
$\mu_1, \ldots, \mu_L$ and the auxiliary space $W$  
with spectral parameter $\lambda$.
If these spaces are labeled as $W_1\otimes \cdots \otimes W_L$ and $W_0$, 
the stochastic $R$ matrix $\mathscr{S}_{0,i}(\lambda, \mu_i)$  
acts as $\mathscr{S}(\lambda, \mu_i)$ 
on $W_0 \otimes W_i$ and as the identity elsewhere.

Thanks to the properties (\ref{sybe}) and (\ref{sinv}),
the matrix (\ref{ngm}) forms a commuting family (cf. \cite{Bax}):
\begin{align}\label{mri}
[T(\lambda|\mu_1,\ldots, \mu_L), 
T(\lambda'|\mu_1,\ldots, \mu_L)]=0.
\end{align}
We write the vector 
$|\alpha_1\rangle  \otimes \cdots \otimes |\alpha_L\rangle  \in W^{\otimes L}$
representing a state of the system 
as $|\alpha_1,\ldots, \alpha_L\rangle$ and the action of 
$T=T(\lambda|\mu_1,\ldots, \mu_L) $ as
\begin{align*}
T|\beta_1,\ldots, \beta_L\rangle 
= \sum_{\alpha_1,\ldots, \alpha_L \in \Z_{\ge 0}^n}
 T_{\beta_1,\ldots, \beta_L}^{\alpha_1,\ldots, \alpha_L}
|\alpha_1,\ldots, \alpha_L\rangle 
\in W^{\otimes L}.
\end{align*}
Then the matrix element is depicted by the concatenation of (\ref{vertex}) as 
\begin{equation}\label{tdiag}
\begin{picture}(250,50)(10,-25)
\put(-20,0){$T_{\beta_1,\ldots, \beta_L}^{\alpha_1,\ldots, \alpha_L}=
{\displaystyle \sum_{\gamma_1,\ldots, \gamma_L \in \Z_{\ge 0}^n}}$}

\put(100,0){
\put(0,0){\vector(1,0){24}}
\put(12,-12){\vector(0,1){24}}
\put(-12,-2){$\gamma_L$}\put(28,-2){$\gamma_1$}
\put(9,-22){$\beta_1$}\put(8,16){$\alpha_1$}}

\put(140,0){
\put(0,0){\vector(1,0){24}}
\put(12,-12){\vector(0,1){24}}
\put(27,-2){$\gamma_2$}
\put(9,-22){$\beta_2$}\put(8,16){$\alpha_2$}}

\put(182,-3){$\cdots$}

\put(220,0){
\put(0,0){\vector(1,0){24}}
\put(12,-12){\vector(0,1){24}}
\put(-24,-2){$\gamma_{L-1}$}\put(27,-2){$\gamma_L$.}
\put(9,-22){$\beta_L$}\put(8,16){$\alpha_L$}}
\end{picture}
\end{equation}
By the construction it satisfies the weight conservation:
\begin{align}\label{wc}
T_{\beta_1,\ldots, \beta_L}^{\alpha_1,\ldots, \alpha_L} = 0
\;\;\text{unless}\;\;
\alpha_1+\cdots +\alpha_L = 
\beta_1+\cdots + \beta_L \in \Z_{\ge 0}^{n}.
\end{align}

Let $t$ be a time variable and consider the evolution equation 
\begin{align}\label{dmt}
|P(t+1)\rangle = T(\lambda|\mu_1,\ldots, \mu_L)
|P(t)\rangle \in W^{\otimes L}.
\end{align}
Although this is an equation in an infinite-dimensional vector space,
the property (\ref{wc}) lets it split into 
finite-dimensional subspaces which we call {\em sectors}.
In terms of the array $m=(m_1,\ldots, m_n) \in \Z^n_{\ge 0}$ 
and the set 
\begin{align*}
S(m) = \{(\sigma_1,\ldots, \sigma_L) \in (\Z^n_{\ge 0})^L\mid 
\sigma_1+\cdots + \sigma_L = m\},
\end{align*}
the corresponding sector, which will also be referred to as $m$, is given by 
$\oplus_{(\sigma_1,\ldots, \sigma_L)\in S(m)}
\C |\sigma_1,\ldots, \sigma_L\rangle$.
We interpret a vector 
$|\sigma_1,\ldots, \sigma_L\rangle \in W^{\otimes L}$ 
with $\sigma_i=(\sigma_{i,1},\ldots, \sigma_{i,n}) \in \Z_{\ge 0}^n$ as 
a state of the system in which the $i$ th site from the left 
is populated with $\sigma_{i,a}$ particles of the $a$ th species.  
Thus $m=(m_1,\ldots, m_n)$ is the multiplicity meaning that 
there are $m_a$ particles of species $a$ in total in the corresponding sector.

In order to interpret (\ref{dmt}) as the master 
equation of a discrete time Markov process, the matrix 
$T=T(\lambda|\mu_1,\ldots, \mu_L) $ should fulfill the following conditions:

\vspace{0.2cm}
\begin{enumerate}
\item  Non-negativity; all the elements (\ref{tdiag}) belong to $\R_{\ge 0}$,

\item Sum-to-unity property; $\sum_{\alpha_1,\ldots, \alpha_L\in \Z_{\ge 0}^{n}}
T_{\beta_1,\ldots, \beta_L}^{\alpha_1,\ldots, \alpha_L} = 1$ for any
$(\beta_1,\ldots, \beta_L) \in (\Z_{\ge 0}^n)^L$.
\end{enumerate}

\vspace{0.2cm}\noindent
The property  (i)  holds if 
$\Phi_q(\gamma|\beta; \lambda,\mu_i)\ge 0$
for all $i \in \Z_L$.
This is achieved 
by taking $0 < \mu^{\epsilon}_i < \lambda ^{\epsilon} < 1, q^{\epsilon}<1$
in the either alternative $\epsilon=\pm 1$.
The property (ii) means the total probability conservation and can be 
shown by using (\ref{syk}) as in \cite[Sec.3.2]{KMMO}.

Henceforth we call the $T(\lambda|\mu_1,\ldots, \mu_L) $
{\em Markov transfer matrix} assuming 
$0 < \mu^{\epsilon}_i < \lambda ^{\epsilon} < 1, q^{\epsilon}<1$ always. 
The choice of $\epsilon=\pm 1$ may be viewed as specifying one 
of the two physical regimes of the system.
The equation (\ref{dmt}) 
represents a stochastic dynamics of $n$-species of particles
hopping to the right periodically via an extra lane (horizontal arrows in (\ref{tdiag}))
which particles get on or get off when they leave or arrive at a site.
The rate of these local processes is specified by 
(\ref{ask2}), (\ref{mho}) and (\ref{vertex}).
For $n=1$ and the homogeneous choice $\mu_1=\cdots= \mu_L$, it reduces to 
the model introduced in \cite{P}.

\begin{example}\label{ex:nmi}
Consider $L=2, n=2$ and the sector $m=(2,1)$, which is the six dimensional space.
We denote the basis in terms of multiset of particles as 
$|12,1\rangle,  |\emptyset,112\rangle$, etc, instead of 
the corresponding multiplicity arrays $|(1,1), (1,0)\rangle, |(0,0),(2,1)\rangle$, etc.
Then $T(\lambda| \mu_1,\mu_2)|1,12\rangle$
is equal to
\begin{align*}
&-\frac{q \mu_1 \mu_2 (\lambda-1)^2 (\lambda-\mu_2) |2,11\rangle}
{(\mu_1-1) (\mu_2-1) \lambda^3 (q \mu_2-1)}
+\frac{\mu_1 (\lambda-1) (\lambda-\mu_2) (\lambda-q \mu_2) |\emptyset,112\rangle}
{(\mu_1-1) (\mu_2-1) \lambda^3 (q \mu_2-1)}\\
&+\frac{\mu_2 (\lambda-1)
   (\lambda-\mu_1) (\lambda-\mu_2) |11,2\rangle}
{(\mu_1-1) (\mu_2-1) \lambda^3 (q \mu_2-1)}
-\frac{\mu_2^2 (\lambda-1) (q \lambda-1) (\lambda-\mu_1) |112,\emptyset\rangle}
{(\mu_1-1) (\mu_2-1) \lambda^3 (q \mu_2-1)}\\
&+\frac{\mu_2 (\lambda-1) \left(q \mu_1 \mu_2
   \lambda^2-q \mu_1 \lambda-q \mu_2 \lambda-q \mu_1 \mu_2 \lambda+q \mu_1 \mu_2+q \lambda^2-\mu_1 \mu_2 \lambda+\mu_1 \mu_2\right) |12,1\rangle}
{(\mu_1-1) (\mu_2-1) \lambda^3 (q \mu_2-1)}\\
&-\frac{(\lambda-\mu_2) \left(-q \mu_2 \lambda+q \mu_1 \mu_2+\mu_1 \mu_2
   \lambda^2-\mu_1 \lambda
-2 \mu_1 \mu_2 \lambda+\mu_1 \mu_2+\lambda^2\right) |1,12\rangle}
{(\mu_1-1) (\mu_2-1) \lambda^3 (q \mu_2-1)}.
\end{align*}
\end{example}

\subsection{\mathversion{bold}Continuous 
time $U_q(A^{(1)}_n)$-ZRP}\label{ss:sra}
From the homogeneous case $\mu_1=\cdots= \mu_L=\mu$ 
of the Markov transfer matrix
$T(\lambda | \mu) = T(\lambda|\mu,\ldots, \mu)$, 
we extract the two ``Hamiltonians" 
by the so called Baxter formula (cf. \cite[Chap. 10.14]{Bax}):
\begin{align}\label{bax}
H^{(1)} = \left.-\epsilon \mu^{-1}
\frac{\partial \log 
T(\lambda | \mu)}{\partial \lambda}\right|_{\lambda=1},
\qquad
H^{(2)} = \left. \epsilon \mu\,
\frac{\partial \log T(\lambda | \mu)}{\partial \lambda} \right|_{\lambda = \mu}.
\end{align}
From (\ref{mri}) the commutativity $[H^{(1)}, H^{(2)}]=0$ follows.
Moreover they both satisfy 
\renewcommand{\labelenumi}{(\roman{enumi})'}

\vspace{0.2cm}
\begin{enumerate}
\item Non-negativity; all the off-diagonal elements are nonnegative,

\item Sum-to-zero property; the sum of elements in any column is zero.
\end{enumerate}
\renewcommand{\labelenumi}{(\roman{enumi})}

\vspace{0.2cm}\noindent
Thus an infinitesimal version of (\ref{dmt})  of the form
\begin{align}\label{ctmp2}
\frac{d}{dt}|P(t)\rangle =H |P(t)\rangle,\qquad
H = a H^{(1)}+ bH^{(2)} \quad (a,b \in \R_{\ge 0})
\end{align}
defines a continuous time integrable Markov process for any $a,b \in \R_{\ge 0}$.
The Markov matrices consist of the pairwise interaction terms as
$H^{(r)} = \sum_{i \in \Z_L} h^{(r)}_{i,i+1}$, where
$h^{(r)}_{i,i+1}$ acts on the $(i,i+1)$th components from the left in $|P(t)\rangle$
as $h^{(r)} \in \mathrm{End}(W\otimes W)$ and as the identity elsewhere.
The local Markov matrices are given by
\begin{align}
h^{(1)} |\alpha, \beta\rangle 
&= -\epsilon\mu^{-1}\sum_{\gamma \in \Z_{\ge 0}^n}
\Phi'_q(\gamma|\alpha; 1,\mu)|\alpha-\gamma, \beta+\gamma\rangle,\label{hee}\\
h^{(2)} |\alpha, \beta\rangle 
&= \epsilon \mu \sum_{\gamma \in \Z_{\ge 0}^n}
\Phi'_q(\beta-\gamma|\beta; \mu,\mu)|\alpha+\gamma, \beta-\gamma\rangle,
\label{hee2}
\end{align}
where $'$ denotes $\frac{\partial}{\partial \lambda}$ 
and the local transition rate is explicitly given by 
\begin{equation}\label{rate}
\begin{split}
-\epsilon\mu^{-1}\Phi'_q(\gamma|\alpha;1, \mu)
&=\epsilon 
\frac{q^{\varphi(\alpha-\gamma,\gamma)}
\mu^{|\gamma|-1}(q)_{|\gamma|-1}}
{(\mu q^{|\alpha|-|\gamma|})_{|\gamma|}}
\prod_{i=1}^n
\binom{\alpha_i}{\gamma_i}_{\!q} \quad (|\gamma|\ge 1),\\
&= -\epsilon\sum_{i=0}^{|\beta|-1}\frac{q^i}{1-\mu q^i} \quad(|\gamma|=0),
\end{split}
\end{equation}
\begin{equation}\label{rate2}
\begin{split}
\epsilon \mu \Phi'_q(\beta-\gamma|\beta; \mu,\mu)
&= \epsilon 
\frac{q^{\varphi(\gamma, \beta-\gamma)}
(q)_{|\gamma|-1}}
{(\mu q^{|\beta|-|\gamma|})_{|\gamma|}}
\prod_{i=1}^n
\binom{\beta_i}{\gamma_i}_{\!q} \quad (|\gamma|\ge 1),\\
&= -\epsilon\sum_{i=0}^{|\beta|-1}\frac{1}{1-\mu q^i} \quad(|\gamma|=0).
\end{split}
\end{equation}

From (\ref{hee}) and (\ref{hee2}), $H^{(1)}$ and $H^{(2)}$ individually  
defines an $n$-species totally asymmetric zero range processes ($n$-TAZRP) 
in which particles hop to the  
to right and to the left neighbor sites 
with the rate (\ref{rate}) and (\ref{rate2}), respectively. 
In the both cases $\gamma=(\gamma_1,\ldots, \gamma_n)$ 
gives the number of particles of species $1,\ldots, n$ 
jumping out the departure site. 
The opposite directional move originates 
in the different behavior of $T(\lambda|\mu)$ 
at the two ``Hamiltonian points" $\lambda = 1$ and $\lambda=\mu$.
The Markov matrix $H$ (\ref{ctmp2}) is a mixture of them yielding an 
$n$-species asymmetric zero range process.
In \cite{KMMO}, Bethe eigenvalues of $T(\lambda|\mu_1,\ldots, \mu_L)$ 
and $H$ have been obtained.

The above integrable Markov processes cover several models considered earlier.
When $(\epsilon,\mu)=(1,0)$ in $H^{(1)}$, 
the nontrivial local transitions in (\ref{rate}) 
are limited to the case $|\gamma| =1$.
So if $\gamma_a=1$ and the other components of $\gamma$ are 0,
the rate (\ref{rate}) becomes  
$q^{\alpha_1+\cdots + \alpha_{a-1}}\frac{1-q^{\alpha_a}}{1-q}$.
This reproduces the $n$-species $q$-boson process 
in \cite{T} whose $n=1$ case further goes back to \cite{SW}.
When $n=1$, the system (\ref{dmt}) was also 
studied in \cite{P,CP,BP}.
The transition rate (\ref{rate}) for $n=1$ and general $\mu$ reproduces the one in 
\cite[p2]{T0} by a suitable adjustment.
When $(\epsilon,\mu,q)=(1,0,0)$ in $H^{(2)}$,
a kinematic constraint 
$\varphi(\gamma, \beta-\gamma) = 
\sum_{1 \le i<j \le n}\gamma_i(\beta_j-\gamma_j)=0$ arises from (\ref{rate2}).
In fact, in order that $\gamma_a>0$ happens,
one must have 
$\gamma_{a+1}=\beta_{a+1}, 
\gamma_{a+2}=\beta_{a+2}, \ldots, \gamma_n = \beta_n$.
It means that larger species particles have the priority to jump out,
which precisely reproduces the $n$-species TAZRP  
explored in \cite{KMO1, KMO2} after reversing the labeling of the 
species $1,2, \ldots, n$ of the particles.

\begin{remark}\label{re:noi}
Denote the Markov matrix $H$ in (\ref{ctmp2}) by 
$H(a,b,\epsilon,q,\mu)$ exhibiting the dependence on the parameters.
Then a duality relation  
$H(a,b,-\epsilon, q^{-1}, \mu^{-1})
= \mathscr{P}H(\mu b, \mu a,\epsilon,q,\mu)\mathscr{P}^{-1}$ holds,
where 
$\mathscr{P} = \mathscr{P}^{-1} 
\in \mathrm{End}(W^{\otimes L})$ is the ``parity" operator 
reversing the sites as
$\mathscr{P}|\sigma_1,\ldots, \sigma_L\rangle = 
|\sigma_L, \ldots, \sigma_1\rangle$
\cite[Remark 9]{KMMO}.
Under the duality, the condition 
$0 < \mu^{\epsilon}, q^{\epsilon}<1$ 
on the parameters is preserved.
\end{remark}

\begin{example}
With the same convention as Example \ref{ex:nmi} we have 
\begin{align*}
H^{(1)}|1,12\rangle
&= -\frac{(2+q-3 q \mu) |1,12\rangle}{(1-\mu) (1-q \mu)}
+\frac{q |12,1\rangle}{1-q \mu}
+\frac{|11,2\rangle}{1-q\mu}
+\frac{(1-q) \mu |112,\emptyset\rangle}{(1-\mu) (1-q \mu)}
+\frac{|\emptyset,112\rangle}{1-\mu},\\
H^{(2)}|1,12\rangle
&= -\frac{(3-\mu-2 q \mu) |1,12\rangle}{(1-\mu) (1-q \mu)}
+\frac{|12,1\rangle}{1-q \mu}
+\frac{q |11,2\rangle}{1-q\mu}
+\frac{(1-q) |112,\emptyset\rangle}{(1-\mu) (1-q \mu)}
+\frac{|\emptyset,112\rangle}{1-\mu}.
\end{align*}
\end{example}

\section{Steady states}\label{sec:hnka}

\subsection{General remarks and examples}

By definition a steady state of the discrete time 
$U_q(A^{(1)}_n)$-ZRP (\ref{dmt})
is a vector $|\overline{P}\rangle \in W^{\otimes L}$ 
such that
\begin{align}\label{yum3}
|\overline{P}\rangle= T(\lambda|\mu_1,\ldots, \mu_L)|\overline{P}\rangle.
\end{align}
The steady state is unique within each sector $m$.
Apart from $m$, 
it depends on $q$ and the inhomogeneity parameters $\mu_1, \ldots, \mu_L$ but
{\em not} on $\lambda$ thanks to the commutativity (\ref{mri}).
Sectors $m=(m_1,\ldots, m_n)$ such that $\forall m_a \ge 1$ are called {\em basic}.
Non-basic sectors are equivalent to a basic sector of some $n'<n$ models
with a suitable relabeling of the species.
Henceforth we concentrate on the basic sectors. 
The coefficient appearing in the expansion
\begin{align*}
|\overline{P}(m)\rangle = \sum_{(\sigma_1,\ldots, \sigma_L) \in S(m)}
{\mathbb P}(\sigma_1,\ldots, \sigma_L)
|\sigma_1,\ldots, \sigma_L\rangle
\end{align*}
is the steady state probability if it is properly normalized as
$\sum_{(\sigma_1,\ldots, \sigma_L) \in S(m)}
{\mathbb P}(\sigma_1,\ldots, \sigma_L) = 1$.
In this paper unnormalized ones will also be refereed to as
steady state probabilities by abuse of terminology.

If the dependence on the inhomogeneity parameters are exhibited as 
${\mathbb P}(\sigma_1,\ldots, \sigma_L; \mu_1, \ldots, \mu_L)$,
we have the cyclic symmetry 
${\mathbb P}(\sigma_1,\ldots, \sigma_L; \mu_1, \ldots, \mu_L)
={\mathbb P}(\sigma_L,\sigma_1, \ldots, \sigma_{L-1}; 
\mu_L, \mu_1, \ldots, \mu_{L-1})$ by the definition.

\begin{example}\label{ex:kan}
For $L=2, n=2$ and the sector $m=(1,1)$, we have
\begin{equation}\label{szk}
\begin{split}
|\overline{P}(1,1)\rangle &= 
\mu _1^2 \left(1-\mu _2\right)\left(1-q\mu _2\right)
 \left(\mu _1+\mu _2-2 \mu _2 \mu _1\right) 
 |\emptyset,12\rangle\\
&+\mu _1 \mu _2 \left(1-\mu _1\right) \left(1-\mu _2\right) 
\left(\mu _1+ q\mu_2-\mu _1\mu _2 -q\mu _1 \mu _2\right) |1,2\rangle
+ \mathrm{cyclic}.
\end{split}
\end{equation}
For $L=3, n=2$ and the sector $m=(1,1)$, we have
\begin{align*}
|\overline{P}(1,1)\rangle &= 
\mu_1^2 \mu _2^2 \left(1-\mu _3\right)\left(1-q\mu _3\right) 
\left(\mu _1 \mu _2+\mu _1 \mu _3+\mu _2 \mu _3-3 \mu _1 \mu _3 \mu _2\right) 
|\emptyset,\emptyset,12)\\
&+ \mu _1^2 \mu _2 \mu_3
\left(1-\mu _2\right) \left(1-\mu _3\right) 
\left(q\mu _1 \mu _2+\mu _1 \mu _3+ \mu _2 \mu _3
-2\mu _1 \mu _2 \mu _3-q \mu _1 \mu _2 \mu _3\right)
|\emptyset,2,1\rangle\\
&+\mu _1^2 \mu _2 \mu_3
\left(1-\mu _2\right) \left(1-\mu _3\right) 
\left(\mu _1 \mu _2+q\mu _1 \mu _3+ q\mu _2 \mu _3
-\mu _1 \mu _2 \mu _3-2q \mu _1 \mu _2 \mu _3\right)
|\emptyset,1,2\rangle+ \mathrm{cyclic}.
\end{align*}
Here $\mathrm{cyclic}$ means the sum of terms obtained by 
the replacement $\mu_j\rightarrow \mu_{j+i}$ and 
$|\sigma_1,\ldots, \sigma_L\rangle \rightarrow 
|\sigma_{1+i},\ldots, \sigma_{L+i}\rangle$ over  
$i \in \Z_L$ with $i \neq 0$.
\end{example}

Let us proceed to the steady states of the continuous time 
$U_q(A^{(1)}_n)$-ZRP in Section \ref{ss:sra}.
By the construction (\ref{bax}) and the commutativity (\ref{mri}), 
the Markov matrices $H^{(1)}, H^{(2)}$ and $H$ share the common 
steady state in each sector.
It is given by specializing the discrete time result 
${\mathbb P}(\sigma_1,\ldots, \sigma_L; \mu_1, \ldots, \mu_L)$ 
to the homogeneous case
$\mu_1=\cdots = \mu_L = \mu$.
For instance, the latter $|\overline{P}(1,1)\rangle$ in Example \ref{ex:kan}
reproduces $|\bar{P}_3\rangle$ in 
\cite[Ex.13]{KMMO} via the specialization $\mu_1=\mu_2 = \mu_3 = \mu$.

\begin{example}\label{ex:lin}
For $L=2, n=2$ and the sector $m=(2,1)$, we have
\begin{align*}
|\overline{P}(2,1)\rangle &= 
(1-q^2\mu)(3+q-\mu-3q\mu)|\emptyset, 112\rangle
+(1-\mu)(1+q+2q^2-2q\mu-q^2\mu-q^3\mu)|2,11\rangle\\
&+(1+q)(1-\mu)(2+q+q^2-\mu-q\mu-2q^2\mu)|1,12\rangle+ \mathrm{cyclic}.
\end{align*}
For $L=3, n=2$ and the sector $m=(2,1)$, we have
\begin{align*}
|\overline{P}(2,1)\rangle &= 
3(1-q\mu)(1-q^2\mu)(2+q-(1+2q)\mu)|\emptyset,\emptyset,112\rangle\\
&+(1-\mu)(1-q\mu)(3+3q+3q^2-(1+5q+2q^2+q^3)\mu)|\emptyset,2,11\rangle\\
&+(1+q)(1-\mu)(1-q\mu)(3+3q+3q^2-(2+2q+5q^2)\mu)|\emptyset,1,12\rangle\\
&+(1+q)(1-\mu)(1-q\mu)(5+2q+2q^2-(3+3q+3q^2)\mu)|\emptyset,12,1\rangle\\
&+(1-\mu)(1-q\mu)(1+2q+5q^2+q^3-(3q+3q^2+3q^3)\mu)|\emptyset,11,2\rangle\\
&+(1+q)(1+q+q^2)(1-\mu)^2(2+q-(1+2q)\mu)|1,1,2\rangle + \mathrm{cyclic}.
\end{align*}
The cyclic here means the sum of terms obtained by 
the replacement 
$|\sigma_1,\ldots, \sigma_L\rangle \rightarrow
|\sigma_{1+i},\ldots, \sigma_{L+i}\rangle$ over  
$i \in \Z_L$ with $i \neq 0$.
\end{example}

Examples \ref{ex:kan} and \ref{ex:lin} indicate that
${\mathbb P}(\sigma_1,\ldots, \sigma_L; \mu_1, \ldots, \mu_L) \in
\Z_{\ge 0}[-\mu_1,\ldots, -\mu_L,q]$ holds in an appropriate normalization. 

\subsection{Matrix product construction}

Let us consider the discrete time $U_q(A^{(1)}_n)$-ZRP in Section \ref{ss:sae}
whose master equation is (\ref{dmt}).
We seek the steady state probability in the matrix product form 
\begin{align}\label{mst}
{\mathbb P}(\sigma_1,\ldots, \sigma_L)
= \mathrm{Tr}(X_{\sigma_1}(\mu_1)\cdots X_{\sigma_L}(\mu_L))
\end{align}
in terms of some operator 
$X_\alpha(\mu)$ with $\alpha \in \Z^n_{\ge 0}$.
Our strategy is to invoke the following result.

\begin{proposition}\label{pr:srn}
Suppose the operators 
$X_\alpha(\mu) \; (\alpha \in \Z^n_{\ge 0})$ obey the relation 
\begin{align}\label{msk}
X_\alpha(\mu)X_\beta(\lambda) = 
\sum_{\gamma,\delta \in \Z^n_{\ge 0}}
\mathscr{S}(\lambda, \mu)^{\beta, \alpha}_{\gamma,\delta}
X_\gamma(\lambda)X_\delta(\mu).
\end{align}
Suppose further that $X_0(\lambda)$ is invertible
and the following auxiliary condition is satisfied:
\begin{align}\label{koi}
X_{\beta}(\mu)X_0(\lambda)^{-1}X_\gamma(\lambda) = 
q^{\varphi(\beta,\gamma)}
\frac{\g_\beta(\mu) \g_\gamma(\lambda)}
{\g_{\beta+\gamma}(\mu)}
X_{\beta+\gamma}(\mu).
\end{align}
Then (\ref{mst}) gives the steady state probability of the system 
(\ref{dmt}) if 
the trace is convergent and not identically zero.
\end{proposition}
\begin{proof}
In view of (\ref{yum3}) we are to show
\begin{align}\label{mdsy}
\mathrm{Tr}(X_{\alpha_1}(\mu_1)\cdots X_{\alpha_L}(\mu_L))
= \sum_{\beta_1,\ldots, \beta_L \in \Z^n_{\ge 0}}
T(\lambda|\mu_1,\ldots, \mu_L)^{\alpha_1,\ldots, \alpha_L}_{\beta_1,\ldots, \beta_L}
\mathrm{Tr}(X_{\beta_1}(\mu_1)\cdots X_{\beta_L}(\mu_L)).
\end{align}
Introduce the elements 
$M_{\gamma; \beta_1,\ldots, \beta_L}^{\alpha_1,\ldots, \alpha_L; \delta}
=M(\lambda|\mu_1,\ldots, \mu_L)_{\gamma; 
\beta_1,\ldots, \beta_L}^{\alpha_1,\ldots, \alpha_L; \delta}$ 
of the monodromy matrix by the diagram similar to (\ref{tdiag}) as
\begin{equation*}
\begin{picture}(250,50)(10,-25)
\put(-30,0){$
M_{\gamma; \beta_1,\ldots, \beta_L}^{\alpha_1,\ldots, \alpha_L; \delta}=
{\displaystyle \sum_{\gamma_1,\ldots, \gamma_{L-1} \in \Z_{\ge 0}^n}}$}

\put(100,0){
\put(0,0){\vector(1,0){24}}
\put(12,-12){\vector(0,1){24}}
\put(-9,-2){$\gamma$}\put(30,-2){$\gamma_1$}
\put(9,-22){$\beta_1$}\put(8,16){$\alpha_1$}}

\put(140,0){
\put(0,0){\vector(1,0){24}}
\put(12,-12){\vector(0,1){24}}
\put(27,-2){$\gamma_2$}
\put(9,-22){$\beta_2$}\put(8,16){$\alpha_2$}}

\put(182,-3){$\cdots$}

\put(220,0){
\put(0,0){\vector(1,0){24}}
\put(12,-12){\vector(0,1){24}}
\put(-24,-2){$\gamma_{L-1}$}\put(27,-3){$\delta$,}
\put(9,-22){$\beta_L$}\put(8,16){$\alpha_L$}}
\end{picture}
\end{equation*}
where the $i$ th vertex from the left denotes the element 
(\ref{vertex}) of $\mathscr{S}(\lambda, \mu_i)$.
By the definition we have
$M(\lambda|\mu_1,\ldots, \mu_L)_{\gamma; 
\beta_1,\ldots, \beta_L}^{\alpha_1,\ldots, \alpha_L; \delta}=0$
unless $\delta+\sum_{i=1}^L\alpha_i=\gamma+\sum_{i=1}^L\beta_i$.
Elements of 
the Markov transfer matrix is given by  
$T(\lambda|\mu_1,\ldots, \mu_L)^{\alpha_1,\ldots, \alpha_L}_{\beta_1,\ldots, \beta_L}
= \sum_{\gamma \in \Z^n_{\ge 0}}
M(\lambda|\mu_1,\ldots, \mu_L)_{\gamma; 
\beta_1,\ldots, \beta_L}^{\alpha_1,\ldots, \alpha_L; \gamma}$, where the sum is 
bounded by $\gamma \le \alpha_1, \beta_L$.
See the remarks after (\ref{mho}).
Now we have
\begin{align*}
\text{LHS of (\ref{mdsy})} &=\mathrm{Tr}(X_{\alpha_1}(\mu_1)\cdots X_{\alpha_L}(\mu_L)X_0(\lambda)X_0(\lambda)^{-1})\\
&\overset{(\ref{msk})}{=} \sum_{\gamma, \beta_1, \ldots, \beta_L}
M_{\gamma; 
\beta_1,\ldots, \beta_L}^{\alpha_1,\ldots, \alpha_L; 0}\;
\mathrm{Tr}(X_{\beta_1}(\mu_1)\cdots 
X_{\beta_{L-1}}(\mu_{L-1})X_{\beta_L}(\mu_L)X_0(\lambda)^{-1}
X_\gamma(\lambda))\\
&\overset{(\ref{koi})}{=} 
\sum_{\gamma, \beta_1, \ldots, \beta_L}
M_{\gamma; 
\beta_1,\ldots, \beta_L}^{\alpha_1,\ldots, \alpha_L; 0}\;
\mathrm{Tr}(X_{\beta_1}(\mu_1)\cdots 
X_{\beta_{L-1}}(\mu_{L-1})X_{\beta_L+\gamma}(\mu_L))
q^{\varphi(\beta_L,\gamma)}
\frac{\g_{\beta_L}(\mu_L) \g_\gamma(\lambda)}
{\g_{\beta_L+\gamma}(\mu_L)}\\
&\overset{(\ref{gkn})}{=} 
\sum_{\gamma, \beta_1, \ldots, \beta_L}
M_{\gamma; 
\beta_1,\ldots, \beta_{L-1}, \beta_L + \gamma}^{\alpha_1,\ldots, \alpha_L; \gamma}\;
\mathrm{Tr}(X_{\beta_1}(\mu_1)\cdots 
X_{\beta_{L-1}}(\mu_{L-1})X_{\beta_L+\gamma}(\mu_L)).
\end{align*}
By replacing $\beta_L+\gamma$ with $\beta_L$ and summing over $\gamma$,
the coefficient of the trace in the last expression becomes 
$T(\lambda|\mu_1,\ldots, \mu_L)^{\alpha_1,\ldots, \alpha_L}_{\beta_1,\ldots, \beta_L}$.
\end{proof}
Observe a perfect fit of the auxiliary condition (\ref{koi}) and 
the property of the stochastic $R$ matrix (\ref{gkn}).

As explained before Example \ref{ex:lin},
a matrix product formula 
for the continuous time $U_q(A^{(1)}_n)$-ZRP in Section \ref{ss:sra} 
follows from (\ref{mst}) just by the specialization $\mu_1=\cdots = \mu_L = \mu$.

The relation of the form (\ref{msk}) 
possessing a solution to the Yang-Baxter equation
as the structure function is often called the
{\em Zamolodchikov-Faddeev} (ZF) {\em algebra}.
In our case 
its associativity is assured by the 
transposed Yang-Baxter equation (\ref{tybe})
rather than (\ref{sybe}).
Proposition \ref{pr:srn} implies that 
the RHS of (\ref{mst}) satisfies the 
Knizhnik-Zamolodchikov type equation (cf. \cite{GDW}) 
as a function of $\mu_1, \ldots, \mu_L$.

From (\ref{ask2}) the ZF algebra (\ref{msk}) reads more explicitly as
\begin{align*}
X_\alpha(\mu)X_\beta(\lambda) = 
\sum_{\gamma \le \alpha}\Phi_q(\beta|\alpha+\beta-\gamma;\lambda, \mu)
X_{\gamma}(\lambda)X_{\alpha+\beta-\gamma}(\mu).
\end{align*}
We find it convenient to work with $Z_\alpha(\mu)$ defined by 
\begin{align}\label{aim}
X_\alpha(\mu) = \g_\alpha(\mu)Z_\alpha(\mu),
\end{align}
where $\g_\alpha(\mu)$ was defined in 
(\ref{mgd}).
Then by using the identity (\ref{oys}), the ZF algebra 
and the auxiliary condition are cast into
\begin{align}
&Z_\alpha(\mu)Z_\beta(\lambda) = 
\sum_{\gamma \le \alpha}
q^{\varphi(\alpha-\gamma, \beta-\gamma)}
\Phi_q(\gamma|\alpha; \lambda,\mu)
Z_\gamma(\lambda)Z_{\alpha+\beta-\gamma}(\mu),\label{mrn}\\
&Z_{\beta}(\mu)Z_0(\lambda)^{-1}Z_\gamma(\lambda) = 
q^{\varphi(\beta,\gamma)}Z_{\beta+\gamma}(\mu). 
\label{aaa}
\end{align}
The expression (\ref{aaa}) 
or equivalently 
$Y_{\beta}(\mu)Y_\gamma(\lambda) = 
q^{\varphi(\beta,\gamma)}Y_{\beta+\gamma}(\mu)$
in terms of the operator 
$Y_\alpha(\mu):= Z_0(0)^{-1}Z_\alpha(\mu)$
is the simplest presentation of the auxiliary condition.

\begin{proposition}\label{pr:irt}
The ZF algebra (\ref{msk}) (or (\ref{mrn})) and the auxiliary condition 
(\ref{koi}) (or (\ref{aaa})) admit a ``trivial representation" 
in terms of an operator $K_\alpha$ 
satisfying 
$K_0 = 1$ and $K_\alpha K_\beta = q^{\varphi(\alpha, \beta)}K_{\alpha+\beta}$ as
\begin{align*}
X_\alpha(\mu) \mapsto \g_\alpha(\mu)K_\alpha,\quad
Z_\alpha(\mu) \mapsto K_\alpha.
\end{align*}
\end{proposition}
\begin{proof}
The auxiliary condition is easily checked. 
The relation (\ref{mrn}) is rewritten as
\begin{align*}
1= \sum_{\gamma \le \alpha}
q^{\varphi(\alpha-\gamma, \beta-\gamma)
-\varphi(\alpha, \beta)+\varphi(\gamma,\alpha+\beta-\gamma) }
\Phi_q(\gamma|\alpha;\lambda, \mu).
\end{align*}
Let $\alpha'=(\alpha_n,\ldots, \alpha_1)$ be the reverse ordered array of 
$\alpha$ as defined after (\ref{mgd}).
Then the obvious identity 
$\varphi(\alpha, \beta) = \varphi(\beta',\alpha')$ leads to 
$\Phi_q(\gamma|\alpha;\lambda, \mu)
= q^{\varphi(\alpha, \gamma)-\varphi(\gamma, \alpha)}
\Phi_q(\gamma'|\alpha';\lambda, \mu)$, hence
the RHS of the above relation is
$\sum_{\gamma \le \alpha}\Phi_q(\gamma'|\alpha';\lambda, \mu)$.
This equals 1 thanks to (\ref{syk}).
\end{proof}

When $n=1$,  $\varphi(\alpha,\beta)=0$ holds 
by (\ref{mho}).
Therefore we may set $Z_\alpha(\mu)=K_\alpha=1$. 
Then (\ref{mst}) gives the result
\begin{align}\label{msk2}
{\mathbb P}(\sigma_1,\ldots, \sigma_L)
= \prod_{i=1}^L \g_{\sigma_i}(\mu_i)
= \prod_{i=1}^L \frac{\mu_i^{-\sigma_i}(\mu_i)_{\sigma_i}}{(q)_{\sigma_i}},
\end{align}
where the $i$ th site variable is a single integer $\sigma_i \in \Z_{\ge 0}$.
In the homogeneous case $\mu_1=\cdots = \mu_L=\mu$,
one can remove the common overall factor within a given sector.
It leads to 
${\mathbb P}(\sigma_1,\ldots, \sigma_L) 
= \prod_{i=1}^L\frac{(\mu)_{\sigma_i}}{(q)_{\sigma_i}}$
up to an overall normalization.
This reproduces the product measure in \cite[eqs. (3), (7)]{P}. 
See also \cite{EMZ}.

It is easy to construct $K_\alpha$ obeying  
$K_\alpha K_\beta = q^{\varphi(\alpha, \beta)}K_{\alpha+\beta}$ 
for general $n$ by using $q$-commuting operators.
However, doing so naively runs into the trouble  
$\mathrm{Tr}(K_{\alpha_1}\cdots K_{\alpha_L})=0$.
The content of the next section grew out of 
an effort to overcome it for $n=2$. 
See also the remark after (\ref{air}).

\section{$U_q(A^{(1)}_2)$ case}\label{sec:rna}
The operator $Z_\alpha(\mu)$ in (\ref{aim})  captures the
multispecies effect beyond the product measure (\ref{msk2}) for $n=1$.
From now on we concentrate on the next nontrivial case $n=2$.
We consider the regime $\epsilon=+1, 0< \forall \mu_i, \mu, q <1$
without losing generality thanks to Remark \ref{re:noi}.

\subsection{\mathversion{bold}$q$-boson realization}
Consider the Fock space 
$F = \bigoplus_{m \ge 0}\C |m\rangle$, its dual
$F^\ast = \bigoplus_{m \ge 0}\C \langle m |$ and the 
operators $\bp, \bm, \bk$ acting on them as
\begin{equation}\label{yrk}
\begin{split}
\bp | m \rangle &= |m+1\rangle,\qquad \bm | m \rangle = (1-q^m)|m-1\rangle,
\qquad \bk |m\rangle = q^m |m \rangle,\\
\langle m | \bm &= \langle m+1 |,\qquad
\langle m | \bp = \langle m-1|(1-q^m),\qquad
\langle m | \bk = \langle m | q^m,
\end{split}
\end{equation}
where $|-1\rangle = \langle -1 |=0$.
They satisfy
\begin{align}\label{akn}
\bk {\bf b}_\pm = q^{\pm 1} {\bf b}_\pm \bk,\qquad
\bp \bm = 1 - \bk,\qquad \bm \bp = 1-q\bk.
\end{align}
We specify the bilinear pairing of $F^\ast$ and $F$ as 
$\langle m | m'\rangle = \theta(m=m')(q)_m$.
Then $\langle m| (X|m'\rangle) = (\langle m|X)|m'\rangle$ holds and the 
trace is given by $\mathrm{Tr}(X) = \sum_{m \ge 0}
\frac{\langle m|X|m\rangle}{(q)_m}$.

Let $\mathcal{B}$ denote the $q$-boson algebra 
generated by $1, \bpm, \bk$ obeying the relations (\ref{akn}).
As a vector space, it has the direct sum decomposition
$\mathcal{B} = \C 1 \oplus \mathcal{B}_{\text{fin}}$,
where $\mathcal{B}_{\text{fin}} = 
\oplus_{r \ge 1}  (\mathcal{B}_+^r \oplus  \mathcal{B}_-^r  
\oplus \mathcal{B}_0^r)$ with
$\mathcal{B}^r_\pm =\oplus_{s\ge 0} \C \bk^s\bpm^r$ and  
$\mathcal{B}^r_0 =\C \bk^r$.
The trace $\mathrm{Tr}(X)$ is convergent if 
$X \in \mathcal{B}_{\text{fin}}$.
It vanishes unless $X \in \oplus_{r \ge 1}\mathcal{B}^r_0$ when it is 
evaluated by $\mathrm{Tr}(\bk^r) = (1-q^r)^{-1}$.
It is an easy exercise to verify the following generalization:
\begin{align}\label{utk}
\mathrm{Tr}(\bk^{m_2}\bm^{m_1}\bp^{m_1}) 
= \frac{(q)_{m_1}(q)_{m_2-1}}{(q)_{m_1+m_2}}\qquad (m_1\ge 0, \;m_2\ge 1).
\end{align}
The trace is invariant under the replacement 
$\bpm \mapsto c^{\pm 1} \bpm$ for any nonzero constant $c$
since it is an automorphism $\mathcal{B}$.
The $q$-boson algebra $\mathcal{B}$ here is slightly different from 
those in \cite[Sec.3.1]{KMO0} and \cite[eq.(3.2)]{KMO2}.

The following result is the main source of our matrix product formula.
\begin{theorem}\label{th:nzm}
For $\alpha = (\alpha_1, \alpha_2) \in \Z_{\ge 0}^2$, the operator
\begin{align}\label{syr}
X_\alpha(\mu) = \g_\alpha(\mu)Z_\alpha(\mu),\quad
Z_\alpha(\mu)
= \frac{(\bp)_\infty}{(\mu^{-1}\bp)_\infty}\bk^{\alpha_2}\bm^{\alpha_1}
\end{align}
satisfies the ZF algebra (\ref{msk}) and the auxiliary condition (\ref{koi}).
\end{theorem}

The proof will be presented in Section \ref{ss:yna}.
The ratio of the infinite products are defined in terms of the series expansion:
\begin{align}\label{air}
\frac{(zw)_\infty}{(z)_\infty}
= \sum_{j \ge 0}\frac{(w)_j}{(q)_j}z^j.
\end{align}

In (\ref{syr}), the factor $K_\alpha = \bk^{\alpha_2}\bm^{\alpha_1}$ 
realizes the trivial representation in Proposition \ref{pr:irt}.
However, taking it only as 
$Z_\alpha(\mu) = \bk^{\alpha_2}\bm^{\alpha_1}$ 
leads to the vanishing trace 
$\mathrm{Tr}(Z_{\sigma_1}(\mu_1)\cdots Z_{\sigma_L}(\mu_L))=0$.
In this sense the representation (\ref{syr}) of the ZF algebra is a
perturbation series from the trivial representation with respect to $\bp$
such that the trace acquires nonzero contribution.

\subsection{Steady state probability}
Let $(\sigma_1,\ldots, \sigma_L) \in S(m)$ be a configuration 
in a basic sector $m=(m_1, m_2) \in \Z_{\ge 1}^2$, where each local state
is the two component array 
$\sigma_i = (\sigma_{i,1},\sigma_{i,2}) \in \Z^2_{\ge 0}$.
Then the formula (\ref{mst}) becomes
\begin{equation}\label{mzs}
\begin{split}
{\mathbb P}(\sigma_1,\ldots, \sigma_L) 
&= \left(\,\prod_{i=1}^L\g_{\sigma_i}(\mu_i) \right)
\mathrm{Tr}(Z_{\sigma_1}(\mu_1)\cdots Z_{\sigma_L}(\mu_L))\\
&=\left(\,\prod_{i=1}^L\frac{\mu_i^{-|\sigma_i|}(\mu_i)_{|\sigma_i|}}
{(q)_{\sigma_{i,1}}(q)_{\sigma_{i,2}}} \right)
\mathrm{Tr}\left(
\frac{(\bp)_\infty}{(\mu_1^{-1}\bp)_\infty}\bk^{\sigma_{1,2}}\bm^{\sigma_{1,1}}
\cdots 
\frac{(\bp)_\infty}{(\mu_L^{-1}\bp)_\infty}\bk^{\sigma_{L,2}}\bm^{\sigma_{L,1}}
\right).
\end{split}
\end{equation}
The element in the trace belongs to (a completion of) $\mathcal{B}_{\text{fin}}$
thanks to $\sigma_{1,2}+ \cdots + \sigma_{L,2} = m_2 \ge 1$.
Thus the trace is convergent and 
(\ref{mzs}) provides a matrix product formula of the
steady state probability. 

\begin{example}\label{ex:aoi}
Consider $L=2$ and the sector $m=(1,1)$.
We calculate (\ref{mzs}) as
\begin{align*}
{\mathbb P}(\emptyset, 12) 
&= \g_{0,0}(\mu_1)\g_{1,1}(\mu_2)
\mathrm{Tr}(Z_{0,0}(\mu_1)Z_{1,1}(\mu_2))= 
\frac{\mu_2^{-2}(\mu_2)_2}{(q)_1^2}\mathrm{Tr}\left(
\frac{(\bp)_\infty}{(\mu_1^{-1}\bp)_\infty}
\frac{(\bp)_\infty}{(\mu_2^{-1}\bp)_\infty}\bk\bm\right)\\
&= \frac{\mu_2^{-2}(\mu_2)_2}{(q)_1^2}\left(\frac{\mu^{-1}_1(\mu_1)_1}{(q)_1}+
\frac{\mu^{-1}_2(\mu_2)_1}{(q)_1}\right)\mathrm{Tr}(\bp\bk\bm),\\
{\mathbb P}(1,2) 
&= \g_{1,0}(\mu_1)\g_{0,1}(\mu_2)
\mathrm{Tr}(Z_{1,0}(\mu_1)Z_{0,1}(\mu_2))
= \frac{(\mu_1\mu_2)^{-1}(\mu_1)_1(\mu_2)_1}{(q)^2_1}
\mathrm{Tr}\left(
\frac{(\bp)_\infty}{(\mu_1^{-1}\bp)_\infty}\bm
\frac{(\bp)_\infty}{(\mu_2^{-1}\bp)_\infty}\bk\right)\\
&= \frac{(\mu_1\mu_2)^{-1}(\mu_1)_1(\mu_2)_1}{(q)^2_1}
\left(\frac{\mu^{-1}_1(\mu_1)_1}{(q)_1}\mathrm{Tr}(\bp\bm\bk)
+ \frac{\mu^{-1}_2(\mu_2)_1}{(q)_1}\mathrm{Tr}(\bm\bp\bk)\right)
\end{align*}
by means of (\ref{air}).
In view of 
$\mathrm{Tr}(\bp\bk\bm) = q^{-1}\mathrm{Tr}(\bp\bm\bk)
= \mathrm{Tr}(\bm\bp\bk) = (1-q^2)^{-1}$, 
these results coincide with (\ref{szk}) if they are commonly multiplied by 
$(\mu_1\mu_2)^3(q)_2(1-q)^2$ which is symmetric in $\mu_1$ and $\mu_2$.
\end{example}

One way to generally characterize 
our normalization of ${\mathbb P}(\sigma_1,\ldots, \sigma_L)$
implied by the formula (\ref{mzs}) is the following calculation
extending the first case of Example \ref{ex:aoi}.

\begin{example}\label{ex:skb}
For general system size $L$ and a general basic sector 
$m=(m_1, m_2) \in \Z^2_{\ge 1}$, 
consider the most condensed configuration 
in which all the particles are contained in a particular site $i$.
Namely $\sigma_j=m$ if $j=i$ and $\emptyset$ otherwise.
Then we have
\begin{align*}
{\mathbb P}(\emptyset,\ldots,\overset{i}{m},\ldots, \emptyset)
&= \frac{\mu_i^{-m_1-m_2}(\mu_i)_{m_1+m_2}}{(q)_{m_1}(q)_{m_2}}
\mathrm{Tr}\left(\bk^{m_2}\bm^{m_1}
\frac{(\bp)_\infty}{(\mu_1^{-1}\bp)_\infty}\cdots
\frac{(\bp)_\infty}{(\mu_L^{-1}\bp)_\infty}\right)\\
&= 
\frac{\mu_i^{-m_1-m_2}(\mu_i)_{m_1+m_2}}{(q)_{m_1+m_2}(1-q^{m_2})}
\sum_{r_1+\cdots+ r_L=m_1}
\frac{(\mu_1)_{r_1}\cdots (\mu_L)_{r_L}}
{(q)_{r_1}\cdots (q)_{r_L}}
\mu_1^{-r_1}\cdots \mu_L^{-r_L},
\end{align*}
where (\ref{utk}) has been used and 
the sum extends over $r_1,\ldots, r_L \in \Z_{\ge 0}$
under the specified condition.
\end{example}

We note that (\ref{mzs}) is convergent also at
$m_1=\sigma_{1,1}+\cdots + \sigma_{L,1}=0$ which is 
outside the basic sector, and reproduces the $n=1$ result 
(\ref{msk}) up to an overall normalization as long as $m_2 \ge 1$.

So far we have treated the inhomogeneous discrete time 
$U_q(A^{(1)}_2)$-ZRP in 
Section \ref{ss:sae}.
A matrix product formula for  
continuous time $U_q(A^{(1)}_2)$-ZRP in Section \ref{ss:sra} 
is obtained from (\ref{mzs}) by the specialization $\mu_1=\cdots = \mu_L = \mu$.
Making the replacement $\bpm \rightarrow \mu^{\pm 1} \bpm$
(see the remark after (\ref{utk})) and 
removing a further common factor within a sector, we arrive at
\begin{equation}\label{msm}
\begin{split}
{\mathbb P}(\sigma_1,\ldots, \sigma_L) 
&=\left(\,\prod_{i=1}^L\frac{(\mu)_{\sigma_{i,1}+\sigma_{i,2}}}
{(q)_{\sigma_{i,1}}(q)_{\sigma_{i,2}}} \right)
\mathrm{Tr}\left(
\frac{(\mu \bp)_\infty}{(\bp)_\infty}\bk^{\sigma_{1,2}}\bm^{\sigma_{1,1}}
\cdots 
\frac{(\mu \bp)_\infty}{(\bp)_\infty}\bk^{\sigma_{L,2}}\bm^{\sigma_{L,1}}
\right).
\end{split}
\end{equation}

\begin{example}\label{ex:naikk}
Consider the homogeneous ZRP $\mu_1=\cdots = \mu_L = \mu$.
Suppose $m=(m_1,m_2)\ge l=(l_1,l_2) \in \Z^2_{\ge 0}$ and 
consider the state 
$(m-l, \emptyset, \ldots,\overset{j}{ l}, \ldots,\emptyset)$
for some $j \in [2,L]$.
It is a less condensed state than Example \ref{ex:skb} 
where $|l|=l_1+l_2$ particles have been separated from the 1 st to the $j$ th site.
One of them is assumed to be the 1 st site without losing generality 
by the $\Z_L$-cyclic symmetry.
Based on computer calculation of (\ref{msm}) we conjecture 
\begin{align}\label{ldma}
\mathbb{P}(m, \emptyset,\ldots, \emptyset)^{-1}
\sum_{|l|=r,\, l \le m}\mathbb{P}(m-l,
\emptyset,\ldots, \overset{j}{l},\ldots, \emptyset)
=\frac{f_{|m|-r}f_r}{f_{|m|}},\qquad
f_s = \frac{(\mu)_s}{(q)_s}
\end{align}
for any site $ j \in [2,L]$ and $0 \le r \le |m|$.
The independence on $L, j$ and $m_1-m_2$ is curious. 
The conjecture has been proven for $r=1$.
See Section \ref{sec:tmn} for more comments. 
\end{example}

At $q=\mu=0$ (and $a=0$ in (\ref{ctmp2})), 
the present model reduces to the $n=2$ 
case of the $n$-TAZRP studied  in \cite{KMO1,KMO2}.
The formula (\ref{msm}) simplifies to 
\begin{align}\label{ntk}
{\mathbb P}(\sigma_1,\ldots, \sigma_L)  = 
\mathrm{Tr}(\tilde{Z}_{\sigma_1} \cdots \tilde{Z}_{\sigma_L}),\qquad
\tilde{Z}_{\sigma_i} = \tilde{Z}_{\sigma_{i,1}, \sigma_{i,2}} = 
\Bigl(\sum_{j \ge 1}\bp^j \Bigr)\bk^{\sigma_{i,2}}\bm^{\sigma_{i,1}}.
\end{align}
This result agrees with \cite{KMO1,KMO2}.
In fact $X_{\alpha_1,\alpha_2}(1)$ in \cite[Example 2.1]{KMO2}
coincides with $\tilde{Z}_{\alpha_2,\alpha_1}$ here.

The matrix product formulas (\ref{mzs}) and (\ref{msm}) 
for the steady state probabilities, 
which are corollaries of Proposition \ref{pr:srn} and Theorem \ref{th:nzm},  
are our main results on the $U_q(A^{(1)}_2)$-ZRP in this paper.

\subsection{Proof of Theorem \ref{th:nzm}}\label{ss:yna}

From $X_0(\mu) = \frac{(\bp)_\infty}{(\mu^{-1}\bp)_\infty}$,
its inverse $X_0(\mu)^{-1} = \frac{(\mu^{-1}\bp)_\infty}{(\bp)_\infty}$ 
certainly exists.
The auxiliary condition (\ref{koi}) is also straightforward to check.
In what follows we shall focus on the proof of the relation (\ref{mrn})
among the $Z_\alpha(\mu)$ specified in (\ref{syr}).
We use a subsidiary variable $\nu= \mu\lambda^{-1}$ throughout.
Substituting (\ref{syr}) into (\ref{mrn}) and using the relations 
\begin{align*}
\bk \frac{(\eta \bp)_\infty}{(\zeta\bp)_\infty} = 
\frac{(q\eta \bp)_\infty}{(q\zeta\bp)_\infty}\bk,\qquad 
\left[ \bm,  \frac{(\eta \bp)_\infty}{(\zeta\bp)_\infty}\right]
= (\zeta-\eta) \frac{(q\eta \bp)_\infty}{(\zeta\bp)_\infty}\bk,
\end{align*}
one can remove the ratio of infinite products.
The result reads
\begin{align}
&(-1)^{\alpha_1}q^{\frac{1}{2}\alpha_1(\alpha_1-1)}
(\lambda^{-1}\bp)_{\alpha_2}(q^{1-\alpha_1}W)_{\alpha_1} \nonumber\\
&\quad= \sum_{\gamma \le \alpha}(-1)^{\gamma_1}
q^{(\gamma_1-\alpha_1)\alpha_2+\frac{1}{2}\gamma_1(\gamma_1-1)}
\nu^{|\gamma|}
\frac{(\lambda)_{|\gamma|}(\nu)_{|\alpha|-|\gamma|}}
{(\mu)_{|\alpha|}}
\binom{\alpha_1}{\gamma_1}_{\!q}\binom{\alpha_2}{\gamma_2}_{\!q} \nonumber\\
&\qquad\qquad \times
(\mu^{-1}\bp)_{\gamma_2}(q^{|\gamma|}\bp)_{|\alpha|-|\gamma|}
(q^{1-\gamma_1}X_{\gamma_2})_{\gamma_1}\bm^{\alpha_1-\gamma_1},
\label{cie}
\end{align}
where 
$W = q^{-\alpha_2}\bm+\lambda^{-1}\bk$ and 
$X_{\gamma_2} = q^{-\gamma_2}\bm + \mu^{-1}\bk$.
Curiously $\mu$ is contained in the RHS only.
In what follows we prove (\ref{cie}) by induction on $\alpha_1$ 
utilizing the following remark. 

\begin{remark}
Suppose a relation $F(\bp, \bm, \bk)=0$ holds in 
the $q$-boson algebra $\mathcal{B}$.
Then $F(c\bp, c^{-1}\bm, \bk)=0$ also holds for any $c \neq 0$ since 
$\bpm \mapsto c^{\pm 1}\bpm$ is an automorphism of $\mathcal{B}$.
\end{remark}

\begin{lemma}\label{le:hrk2}
The relation (\ref{cie}) is valid at $\alpha_1=0$, namely the following holds: 
\begin{align}\label{yum}
(\lambda^{-1}\bp)_{\alpha_2} = 
\sum_{0 \le \gamma_2 \le \alpha_2}
\nu^{\gamma_2}
\frac{(\lambda)_{\gamma_2}(\nu)_{\alpha_2-\gamma_2}}
{(\mu)_{\alpha_2}}\binom{\alpha_2}{\gamma_2}_{\!q} 
(\mu^{-1}\bp)_{\gamma_2}(q^{\gamma_2}\bp)_{\alpha_2-\gamma_2}.
\end{align}
\end{lemma}
\begin{proof}
One can prove it as an identity of polynomials of order $\alpha_2$ in $\bp$.
At $\bp = 0$ it is equivalent to
\begin{align*}
\frac{(\mu)_{\alpha_2}}{(q)_{\alpha_2}} 
= \sum_{i+j=\alpha_2}
\frac{\nu^{j}(\lambda)_{j}(\nu)_i}{(q)_j(q)_i}.
\end{align*}
Due to (\ref{air}) this indeed holds since the RHS is the coefficient of $z^{\alpha_2}$ in
$\frac{(\nu z)_\infty(\mu z)_\infty}{(z)_\infty(\nu z)_\infty}
=\frac{(\mu z)_\infty}{(z)_\infty}$.
It suffices to check (\ref{yum}) further at 
$\bp = \lambda q^{-k}$ with $k=0,1,\ldots, \alpha_2-1$.
By the identity
$(\lambda)_{\gamma_2}(\lambda q^{\gamma_2-k})_{\alpha_2-\gamma_2}
=(\lambda q^{-k})_{\alpha_2}(\lambda q^{\gamma_2-k})_k/(\lambda q^{-k})_k$,
the relation to show takes the form
\begin{align*}
0 &= \sum_{0 \le j \le \alpha_2}\nu^{j}
\left(\nu\right)_{\alpha_2-j}
\left(\nu^{-1} q^{-k}\right)_j
\binom{\alpha_2}{j}_{\!q}
(\lambda q^{j-k})_k\quad (0 \le k < \alpha_2).
\end{align*}
By expanding the last factor into powers of $\lambda$, 
the RHS is expressed as a sum 
$\sum_{0 \le s \le k}\lambda^s A_{k,s} f_{\alpha_2, k,s}(\nu)$ with 
\begin{align*}
f_{a,k,s}(\nu) &= 
 \sum_{0 \le j \le a}(q^s\nu)^j (\nu)_{a-j}(\nu^{-1}q^{-k})_j\binom{a}{j}_{\!q}
 \qquad (0 \le s \le k < a),\\
 A_{k,s}&=(-1)^s q^{\frac{1}{2}s(s-1)-ks}\binom{k}{s}_q.
\end{align*}
Consider the identity
\begin{align*}
(q\nu z)_s(q^{-k+s+1}z)_{k-s} = 
\frac{(q^{-k+s+1}z)_\infty(q\nu z)_\infty}{(q^{s+1}\nu z)_\infty (qz)_\infty}
= \sum_{a\ge 0}\frac{(qz)^af_{a,k,s}(\nu)}{(q)_a},
\end{align*}
where the last equality is due to (\ref{air}).
Since the LHS is an order $k$ polynomial in $z$, it follows that
$f_{a,k,s}(\nu)=0$ for $k<a$.
\end{proof}

\begin{lemma}\label{le:hrk}
Set $Y=\bm+\bk$. 
Then the following equality is valid for any $m \in \Z_{\ge 0}$.
\begin{align*}
(\bp)_m\bm^m = (-1)^mq^{\frac{1}{2}m(m-1)}
\sum_{s=0}^m \mu^{m-s}\binom{m}{s}_{\!q}(\mu)_s(\mu^{-1}q^{1-m}Y)_{m-s}.
\end{align*}
\end{lemma}
\begin{proof}
It is equivalent to
\begin{align*}
(\bp-1)(\bp-q^{-1})\cdots (\bp-q^{1-m}) \bm^m = 
(q)_m\sum_{s+t=m}\frac{(\mu)_s}{(q)_s}
\frac{(\mu^{-1}q^{1-m}Y)_t }{(q)_t}\mu^t,
\end{align*}
where the sum extends over $s, t \in \Z_{\ge 0}$
under the specified constraint.
Applying the identity
\begin{align*}
(\bp-q^{-n+1})\bm^n = (1-\bk-q^{-n+1}\bm)\bm^{n-1}=\bm^{n-1}(1-q^{-n+1}Y)
\end{align*}
successively, one finds that the LHS is equal to $(q^{1-m}Y)_m$.
On the other hand from (\ref{air}), the RHS is equal to the coefficient of 
$z^m$ in the power series
\begin{align*}
(q)_m \frac{(z\mu)_\infty}{(z)_\infty}
\frac{(zq^{1-m}Y)_\infty}{(z\mu)_\infty}
=(q)_m \frac{(zq^{1-m}Y)_\infty}{(z)_\infty} = 
(q)_m \sum_{j \ge 0}\frac{(q^{1-m}Y)_j}{(q)_j}z^j.
\end{align*}
\end{proof}

\vspace{0.2cm}{\em Proof of (\ref{cie}) for general $\alpha_1$}.
Let us write (\ref{cie}) as 
$\mathscr{L}(\bpm, \alpha_1,\alpha_2; \lambda) = 
\mathscr{R}(\bpm, \alpha_1,\alpha_2; \lambda, \mu)$.
We are going to prove it by induction on $\alpha_1$.
The case $\alpha_1=0$ was shown in Lemma \ref{le:hrk2}.
In Lemma \ref{le:hrk} with $m=\gamma_1$, replace
$\bpm$ by $(\mu^{-1}q^{\gamma_2})^{\pm 1}\bpm$.
Then $(\mu^{-1}q^{1-m}Y)_{m-s}$ becomes 
$(q^{1-\gamma_1}X_{\gamma_2})_{\gamma_1-s}$.
Solving it for the $s=0$ term we get
\begin{align*}
(q^{1-\gamma_1}X_{\gamma_2})_{\gamma_1}
= (-\mu)^{-\gamma_1}q^{-\frac{1}{2}\gamma_1(\gamma_1-1)}
(\mu^{-1}q^{\gamma_2}\bp)_{\gamma_1}(\mu q^{-\gamma_2}\bm)^{\gamma_1}
-\sum_{s=1}^{\gamma_1}\mu^{-s}\binom{\gamma_1}{s}_{\!q}
(\mu)_s(q^{1-\gamma_1}X_{\gamma_2})_{\gamma_1-s}.
\end{align*}
Substituting this into the RHS of (\ref{cie}) we have the decomposition
$\mathscr{R}(\bpm, \alpha_1,\alpha_2; \lambda, \mu) =
\mathscr{R}_0(\bpm, \alpha_1,\alpha_2; \lambda, \mu)-
\sum_{s=1}^{\gamma_1}\mathscr{R}_s(\bpm, \alpha_1,\alpha_2; \lambda, \mu)$,
where
\begin{align}
\mathscr{R}_0(\bpm, \alpha_1,\alpha_2; \lambda, \mu) &=
\sum_{\gamma \le \alpha}q^{(\gamma_1-\alpha_1)\alpha_2-\gamma_1\gamma_2}
\nu^{|\gamma|}
\frac{(\lambda)_{|\gamma|}(\nu)_{|\alpha|-|\gamma|}}
{(\mu)_{|\alpha|}}
\binom{\alpha_1}{\gamma_1}_{\!q}\binom{\alpha_2}{\gamma_2}_{\!q}
(\mu^{-1}\bp)_{|\gamma|}(q^{|\gamma|}\bp)_{|\alpha|-|\gamma|}\bm^{\alpha_1},\nonumber\\
\mathscr{R}_s(\bpm, \alpha_1,\alpha_2; \lambda, \mu) &=
\sum_{\gamma \le \alpha}q^{(\gamma_1-\alpha_1)\alpha_2}
\nu^{|\gamma|}
\frac{(\lambda)_{|\gamma|}(\nu)_{|\alpha|-|\gamma|}}
{(\mu)_{|\alpha|}}
\binom{\alpha_1}{\gamma_1}_{\!q}\binom{\alpha_2}{\gamma_2}_{\!q}
(\mu^{-1}\bp)_{\gamma_2}(q^{|\gamma|}\bp)_{|\alpha|-|\gamma|}\nonumber\\
&\qquad\qquad\qquad\times 
(-1)^{\gamma_1}q^{\frac{1}{2}\gamma_1(\gamma_1-1)}
\mu^{-s}
\binom{\gamma_1}{s}_{\!q}
(\mu)_s
(q^{1-\gamma_1}X_{\gamma_2})_{\gamma_1-s}
\bm^{\alpha_1-\gamma_1}.\nonumber
\end{align}
In $\mathscr{R}_0(\bpm, \alpha_1,\alpha_2; \lambda, \mu)$, replace 
$\gamma_2$ by $\gamma_2-\gamma_1$.
Then the sum over $\gamma_1$ can be taken by means of 
$\binom{|\alpha|}{\gamma_2}_{\!q} = \sum_{0 \le \gamma_1 \le \min(\gamma_2,\alpha_1)}
q^{\gamma_1(\alpha_2-\gamma_2+\gamma_1)}
\binom{\alpha_1}{\gamma_1}_{\!q}\binom{\alpha_2}{\gamma_2-\gamma_1}_{\!q}$,
yielding
\begin{align}
\mathscr{R}_0(\bpm, \alpha_1,\alpha_2; \lambda, \mu) 
&= q^{-\alpha_1\alpha_2}
\sum_{\gamma_2\le \alpha_2}
\nu^{\gamma_2}
\frac{(\lambda)_{\gamma_2}(\nu)_{|\alpha|-\gamma_2}}
{(\mu)_{|\alpha|}}
\binom{|\alpha|}{\gamma_2}_{\!q}
(\mu^{-1}\bp)_{\gamma_2}(q^{\gamma_2}\bp)_{|\alpha|-\gamma_2}\bm^{\alpha_1}\nonumber\\
&= q^{-\alpha_1\alpha_2}(\lambda^{-1}\bp)_{|\alpha|}\bm^{\alpha_1}\nonumber\\
&= (-1)^{\alpha_1}q^{\frac{1}{2}\alpha_1(\alpha_1-1)}
(\lambda^{-1}\bp)_{\alpha_2}
\sum_{t=0}^{\alpha_1}\lambda^{-t}(\lambda)_t
\binom{\alpha_1}{t}_{\!q}(q^{1-\alpha_1}W)_{\alpha_1-t},
\label{kyk}
\end{align}
where the second equality is due to Lemma \ref{le:hrk2}.
The third equality is obtained by applying Lemma \ref{le:hrk}
with $m, \mu$ and  $\bpm$ replaced by 
$\alpha_1, \lambda$ and  $(q^{\alpha_2}\lambda^{-1})^{\pm1}\bpm$, respectively.

To evaluate $\mathscr{R}_s(\bpm, \alpha_1,\alpha_2; \lambda, \mu)$ with $s\ge 1$, 
rewrite $\binom{\alpha_1}{\gamma_1}_{\!q}\binom{\gamma_1}{s}_{\!q}$ as
$\binom{\alpha_1-s}{\gamma_1-s}_{\!q}\binom{\alpha_1}{s}_{\!q}$ and then 
change $\gamma_1$ into $\gamma_1+s$.
By this procedure, the formula for 
$\mathscr{R}(\bpm, \alpha_1,\alpha_2; \lambda, \mu)$ gets replaced by
\begin{align}\label{obt}
\mathscr{R}(\bpm, \alpha_1,\alpha_2; \lambda, \mu)
= \mathscr{R}_0(\bpm, \alpha_1,\alpha_2; \lambda, \mu)-
\sum_{s=1}^{\alpha_1}\mathscr{R}_s(\bpm, \alpha_1,\alpha_2; \lambda, \mu).
\end{align}
As for the summand, comparing the resulting expression with 
$\mathscr{R}(q^{\pm s}\bpm, \alpha_1-s, \alpha_2; q^s\lambda, q^s\mu)$, we find
\begin{align}
\mathscr{R}_s(\bpm, \alpha_1,\alpha_2; \lambda, \mu) 
&= 
(-\lambda^{-1})^s(\lambda)_sq^{s\alpha_1-\frac{1}{2}s(s+1)}
\binom{\alpha_1}{s}_{\!q}
\mathscr{R}(q^{\pm s}\bpm, \alpha_1-s, \alpha_2;q^s\lambda, q^s\mu)\nonumber\\
&= (-\lambda^{-1})^s(\lambda)_sq^{s\alpha_1-\frac{1}{2}s(s+1)}
\binom{\alpha_1}{s}_{\!q}
\mathscr{L}(q^{\pm s}\bpm, \alpha_1-s, \alpha_2; q^s\lambda)\nonumber\\
&= (-1)^{\alpha_1}q^{\frac{1}{2}\alpha_1(\alpha_1-1)}(\lambda^{-1}\bp)_{\alpha_2}
\lambda^{-s}(\lambda)_s\binom{\alpha_1}{s}_{\!q}
(q^{1-\alpha_1}W)_{\alpha_1-s},
\label{mkr}
\end{align}
where the second equality is due to the induction assumption.
Now from (\ref{kyk}) and (\ref{mkr}) we see that the difference (\ref{obt}) 
leaves the $t=0$ term of (\ref{kyk}) only,  which exactly coincides with 
$\mathscr{L}(\bpm,\alpha_1,\alpha_2,\lambda)$, i.e., the LHS of (\ref{cie}).
\qed

\section{Summary and discussion}\label{sec:tmn}

We have studied the steady state probabilities of the 
$U_q(A^{(1)}_n)$-ZRPs \cite{KMMO}.
The main results are the attribution to the ZF algebra
and the auxiliary condition for general $n$ (Proposition \ref{pr:srn}), 
a concrete realization of them for $n=2$ (Theorem \ref{th:nzm}) 
and the resulting matrix product formulae in (\ref{mzs}) and (\ref{msm}).

They serve as a starting point for studying physical properties of the system. 
For instance, the RHS of (\ref{ldma}) 
is viewed as a naive measure of the {\em condensation} (cf. \cite{EH,GSS}).
Apart from the statistical factors $L$ and $L(L-1)/2$ 
for the relevant configurations,
the crude estimation
\begin{align*}
\frac{f_{|m|-r}f_r}{f_{|m|}} \simeq \frac{1-\mu}{1-q}\;\;(q\searrow 0),\qquad
\frac{f_{|m|-r}f_r}{f_{|m|}} \simeq 
\exp\left(\Bigl(\log\frac{1-q}{1-\mu}\Bigl)
\frac{(1-q^{r})(1-q^{|m|-r})}{\log q}\right)\;\;
(q \nearrow 1)
\end{align*} 
for $0 < r < |m|$ indicates that particles are more likely to condense
in the region $\mu > q$ than $\mu<q$. 

Another commonly undertaken approach is to switch to the  
grand canonical picture and  
investigate the generating function involving ``fugacity" $x, y$:
\begin{align*}
\mathcal{Z}(w,x,y) &=
\sum_{\sigma_1,\ldots, \sigma_L \in \Z^2_{\ge 0}}
\mathbb{P}_w(\sigma_1,\ldots, \sigma_L)x^{\sigma_{1,1}+\cdots + \sigma_{L,1}}
y^{\sigma_{1,2}+\cdots + \sigma_{L,2}}.
\end{align*}
Here $\mathbb{P}_w(\sigma_1,\ldots, \sigma_L)$ is  
a regularization of (\ref{mzs})  avoiding the
divergence at the non-basic sector $\sigma_{1,2}= \cdots = \sigma_{L,2}=0$.
An example of such a prescription is to insert a
boson-counter ${\bf h} = \log_q \bk$ (see (\ref{yrk})) into the trace as 
$\mathrm{Tr}\left(w^{\bf h} (\cdots) \right)$.
Then (\ref{mzs}) allows one to express it as
\begin{align*}
\mathcal{Z}(w,x,y) &= 
\mathrm{Tr}\left(w^{\bf h} V(\mu_1,x,y)\cdots V(\mu_L,x,y)\right),\\
V(\mu,x,y)& =\frac{(\bp)_\infty}{(\mu^{-1}\bp)_\infty}
\sum_{l,m\ge 0}
\frac{x^m y^l \mu^{-l-m}(\mu)_{l+m}}{(q)_l(q)_m}
\bk^l \bm^m
=\frac{(\bp)_\infty}{(\mu^{-1}\bp)_\infty}
\sum_{m\ge 0}
\frac{(x\mu^{-1})^{m}(\mu)_m}{(q)_m} 
\frac{(q^my\bk)_\infty}{(y\mu^{-1}\bk)_\infty}\bm^m\\
&= \frac{(\bp)_\infty}{(\mu^{-1}\bp)_\infty}
\Gamma(x\mu^{-1}, y\mu^{-1})^{-1}\Gamma(x,y),\qquad
\Gamma(x,y) = (x \bm)_\infty(y \bk)_\infty.
\end{align*}
Similarly in the 
homogeneous case $\mu_1=\cdots = \mu_L=\mu$, the result 
(\ref{msm}) 
corresponding to another normalization of 
$\mathbb{P}_w(\sigma_1,\ldots, \sigma_L)$
leads to the alternative form
\begin{align*}
\mathcal{Z}(w,x,y) = \mathrm{Tr}\left(w^{\bf h}\tilde{V}(\mu,x,y)^L\right),
\quad
\tilde{V}(\mu,x,y) =  \frac{(\mu \bp)_\infty}{(\bp)_\infty}
\Gamma(x,y)^{-1}\Gamma(x\mu, y\mu).
\end{align*}
This formula remains valid at $\mu=0$ 
and may be useful to extract the large $L$ asymptotics in the 
corresponding $(n=2)$-species $q$-boson model \cite{T}.

As noted in (\ref{ntk}), the formula (\ref{msm}) with $q=\mu=0$ agrees 
with the earlier result on the 
$(n=2)$-TAZRP \cite{KMO1, KMO2}
based on the combinatorial $R$ and 
the tetrahedron equation.
It is an interesting question how the approaches in the present paper 
and \cite{KMO1, KMO2} are related for general $n$.
We plan to address it in a future publication.

\section*{Acknowledgments}
The authors thank 
Ivan Corwin, Jan de Gier, Thomas Lam and Kirone Mallick 
for inspiring lectures at  Infinite Analysis 16 conference,  
{\em New Developments in Integrable Systems}, held 
at Osaka City University during 24-27 March 2016.
Thanks are also due to 
Vladimir Mangazeev and Shouya Maruyama
for collaboration in the previous work and 
Satoshi Watanabe for a kind interest.
This work is supported by 
Grants-in-Aid for Scientific Research No.~15K04892,
No.~15K13429 and No.~23340007 from JSPS.

\end{document}